\documentclass[reqno]{amsart}

\usepackage[T1]{fontenc}
\usepackage[utf8]{inputenc}
\usepackage[english]{babel}
\usepackage{lmodern}
\usepackage{microtype}
\usepackage{amsmath,amssymb,amsthm,mathtools}
\usepackage{booktabs}
\usepackage{tabularx}
\usepackage{needspace}
\usepackage[table]{xcolor}
\usepackage[hidelinks]{hyperref}

\definecolor{definitionblue}{rgb}{0.00,0.20,0.70}
\definecolor{tableshade}{gray}{0.95}
\newcommand{\defin}[1]{%
  \relax\ifmmode
    \textcolor{definitionblue}{#1}%
  \else
    \textcolor{definitionblue}{\emph{#1}}%
  \fi
}

\newtheorem{theorem}{Theorem}[section]
\newtheorem{lemma}[theorem]{Lemma}
\newtheorem{proposition}[theorem]{Proposition}

\numberwithin{equation}{section}

\DeclareMathOperator{\des}{des}
\DeclareMathOperator{\asc}{asc}
\DeclareMathOperator{\run}{run}
\DeclareMathOperator{\pk}{pk}
\newcommand{\had}{\mathbin{\odot}}
\newcommand{\oeis}[1]{\href{https://oeis.org/#1}{\texttt{#1}}}
\newcommand{\doi}[1]{\href{https://doi.org/#1}{\nolinkurl{doi:#1}}}
\newcommand{\arxiv}[1]{%
  \href{https://arxiv.org/abs/#1}{\nolinkurl{arXiv:#1}}}

\title[Real-rooted Eulerian polynomials]{Real-rooted Eulerian polynomials
  from permutations, words, and paths}
\author[Per Alexandersson]{Per Alexandersson}
\address{Department of Mathematics, Stockholm University,
  SE-106 91 Stockholm, Sweden}
\email{per.w.alexandersson@gmail.com}
\hypersetup{
  pdftitle={Real-rooted Eulerian polynomials from permutations, words,
    and paths},
  pdfauthor={Per Alexandersson},
  pdfsubject={Real-rootedness proofs for Eulerian-type polynomial families},
  pdfkeywords={Eulerian polynomial, descent polynomial, real-rootedness,
    interlacing}
}

\begin{document}
\raggedbottom

\begin{abstract}
  We study seven Eulerian-type polynomial families.  We prove that the
  descent polynomials of derangements are real-rooted, settling the
  derangement part of a conjecture of S.~Fu, Z.~Lin, and J.~Zeng. We also
  resolve the real-rootedness conjecture in OEIS
  A335340, strengthen the known rowwise real-rootedness of an even-top
  descent family to consecutive strict interlacing, and prove real-rootedness,
  consecutive interlacing, and real-rooted gamma-polynomials for U.~Shankar's
  super-Eulerian polynomials.  A differential recurrence gives consecutive
  weak interlacing for ternary words counted by increasing runs.
  We prove stability of the peak-value refinement and consecutive
  interleaving of its positive weighted diagonals, settling a conjecture of
  P.~Alexandersson and O.~Nabawanda. Finally, we prove that the zig-zag Eulerian
  polynomials, which enumerate swaps in alternating permutations, have simple
  negative roots, settling the real-rootedness conjecture of T.~Kyle Petersen
  and Y.~Zhuang.
\end{abstract}

\maketitle

\section{Introduction}
\label{sec:overview}

For a finite set of permutations or words, the distribution of a descent-like
statistic often has a unimodal coefficient sequence. For polynomials with
nonnegative coefficients, real-rootedness implies log-concavity and unimodality,
but it is sensitive to the combinatorial model and usually requires more
structure than an explicit coefficient formula.

Our main result concerns derangements.  S.~Fu, Z.~Lin, and J.~Zeng conjectured
that their descent polynomials are real-rooted
\cite[Conjecture~5.2]{FuLinZeng2018}.  Theorem~\ref{thm:derangementRealRooted}
proves this assertion using a compatible-pair recursion through Eulerian
operators; finite-symbol stability supplies the required interlacing
orientation.  The same method embeds the reciprocal derangement polynomials
in a larger deletion-stable real-rooted family.

The other part of the Fu--Lin--Zeng conjecture concerns descent polynomials
of separable permutations, which avoid the patterns \(2413\) and \(3142\).
In a recent preprint, Z.~Zhang~\cite[Theorems~5.4 and~5.5]{Zhang2026SeparableInterlacing}
proves that these polynomials have simple negative zeros and strictly interlace
consecutively.  He also proves the corresponding statements for their
gamma-polynomials.  These results follow from his strict-interlacing theorem
for size-dependent factorial compression operators
\cite[Theorem~1.1]{Zhang2026SeparableInterlacing}.  The same theorem yields
strict interlacing for the gamma-polynomials of Postnikov's intransitive-tree
polynomials and for the tree polynomials after removing their forced factor
\(t\)~\cite[Theorems~6.4 and~6.5]{Zhang2026SeparableInterlacing}.

For the remaining six families, some results are new and others follow from
earlier work.  Expressing the
cyclic-descent polynomial of northeast paths as a derivative of a type~\(B\)
Narayana polynomial proves Theorem~\ref{thm:cyclicDescents} and settles the
explicit real-rootedness conjecture in OEIS \oeis{A335340}.  For descents whose
top is even, rowwise real-rootedness already follows from work of J.~T.~Hall
and J.~B.~Remmel together with that of J.~Haglund, K.~Ono, and D.~G.~Wagner.
Two differential recurrences prove the apparently unrecorded strict
interlacing of every consecutive pair in
Theorem~\ref{thm:evenTopInterlacing}.

The ordinary Eulerian polynomials have simple negative zeros by a classical
theorem of G.~Frobenius and a gamma-positive expansion due to D.~Foata and
M.-P.~Sch\"utzenberger~\cite{Frobenius1910,FoataSchutzenberger1970}.  U.~Shankar
introduced a power deformation of their recurrence and proved that its rows
are palindromic and log-concave.  His finite data suggested real-rootedness.
Theorem~\ref{thm:superEulerian} proves this for every positive integral power,
gives consecutive interlacing, and extends his gamma-positivity theorem from
the square case to all positive integral powers.  Thus the new content is the
uniform extension beyond the classical first power, not the corresponding
properties of the ordinary Eulerian polynomials.  The proof combines a
normalization by powers of binomial coefficients with Hadamard-product
theorems of P\'olya--Schur--Wagner and Garloff--Wagner.
Proposition~\ref{prop:superEulerianPrefixes} gives an interlacing refinement
obtained by restricting the largest element of the subsets indexing a row.

G.~Cabrele proved the enumeration of ternary words by increasing runs; its
differential recurrence yields the consecutive weak interlacing in
Theorem~\ref{thm:ternaryRuns}, which appears not to have been recorded.

We also consider the ordinary permutation peak polynomials, indexed by OEIS
\oeis{A008303}.  D.~Warren and E.~Seneta proved their real-rootedness and
consecutive interlacing \cite{WarrenSeneta1996}.  In
Theorem~\ref{thm:peakValueStability}, we prove
that the natural multivariate refinement by peak values is stable and that its
positive weighted diagonals interleave consecutively, settling Alexandersson
and Nabawanda's conjecture.

We finish with the zig-zag Eulerian polynomials studied by Coons and
Sullivant~\cite{CoonsSullivant2023} and Petersen and Zhuang~\cite{PetersenZhuang2024},
which count swaps in alternating permutations.
Theorem~\ref{thm:zigzagEulerian} proves that their roots are simple and negative,
settling their real-rootedness conjecture
\cite[Conjecture~6.1(a)]{PetersenZhuang2024}.
We first identify a centered polynomial by finite factor identities, bound its
complex roots in a fixed disk, and then apply a differential-operator transfer.

\Needspace{5\baselineskip}
For the first four families below, we write
\(\defin{D_n(t)},\defin{C_n(t)},\defin{R_n(t)},\defin{E_n^{(l)}(t)}\);
for the next two, we write \(\defin{W_n(t)},\defin{K_n(t)}\), and for
the zig-zag family, \(\defin{H_n(t)}\).
Table~\ref{tab:familySummary} gives representative initial rows together with
recurrences or defining formulas.  The OEIS entries identify the coefficient arrays; the
corresponding sections cite the primary sources for the recurrences.

\begin{table}[ht]
  \centering
  \footnotesize
  \setlength{\tabcolsep}{4pt}
  \renewcommand{\arraystretch}{1.12}
  \begin{tabularx}{\textwidth}{@{}
      >{\raggedright\arraybackslash}p{0.18\textwidth}lX@{}}
    \toprule
    Family & OEIS & Initial rows and recurrence \\
    \midrule
    Derangement descents
      & \oeis{A219836}
      & \(D_1=0\), \(D_2=t\), \(D_3=2t\). \\
      & & \centering
        \(\displaystyle
        D_n=(1+(n-1)t)D_{n-1}+t(1-t)D_{n-1}'
        +(-1)^nt^{n-1}\quad(n\geq2).\) \tabularnewline[3pt]
    \rowcolor{tableshade}
    Cyclic path descents
      & \oeis{A335340}
      & \(C_1=2t\), \(C_2=4t+2t^2\),
        \(C_3=6t+12t^2+2t^3\). \\
    \rowcolor{tableshade}
      & & \centering
        \(\displaystyle
        C_n=\frac{2t}{n}\frac{d}{dt}
        \sum_{k=0}^n\binom nk^2t^k\quad(n\geq1).\)
        \tabularnewline[3pt]
    Even-top descents
      & \oeis{A134434}
      & \(R_2=1+t\), \(R_3=4+2t\). \\
      & & \centering
        \(\displaystyle
        \begin{aligned}
          R_{2n+1}&=(1-t)R_{2n}'+(2n+1)R_{2n},\\
          R_{2n+2}&=t(1-t)R_{2n+1}'
            +(1+(2n+1)t)R_{2n+1}
          \qquad(n\geq1).
        \end{aligned}\) \tabularnewline[3pt]
    \rowcolor{tableshade}
    Super-Eulerian polynomials
      & ---
      & \(E_1^{(2)}=1\), \(E_2^{(2)}=1+t\),
        \(E_3^{(2)}=1+8t+t^2\). \\
    \rowcolor{tableshade}
      & & \centering
        \(\displaystyle
        E_{n,k}^{(l)}=(k+1)^lE_{n-1,k}^{(l)}
        +(n-k)^lE_{n-1,k-1}^{(l)}\quad(n\geq2).\)
        \tabularnewline[3pt]
    Ternary increasing runs
      & \oeis{A120987}
      & \(W_0=1\), \(W_1=3t\), \(W_2=3t+6t^2\). \\
      & & \centering
        \(\displaystyle
        (n+1)W_{n+1}=(-n+(4n+3)t)W_n
        +3t(1-t)W_n'\quad(n\geq0).\) \tabularnewline[3pt]
    \rowcolor{tableshade}
    Permutation peaks
      & \oeis{A008303}
      & \(K_1=1\), \(K_2=2\), \(K_3=4+2t\),
        \(K_4=8+16t\). \\
    \rowcolor{tableshade}
      & & \centering
        \(\displaystyle
        K_n=(2+(n-2)t)K_{n-1}
        +2t(1-t)K_{n-1}'\quad(n\geq2).\) \tabularnewline
    Zig-zag Eulerian polynomials
      & \oeis{A205497}
      & \(H_1=H_2=1\), \(H_3=1+t\), \(H_4=1+3t+t^2\). \\
      & & \(H_n(t)=\sum_{\sigma\in\mathcal A_n}
        t^{\operatorname{swaps}(\sigma)}\). \tabularnewline
    \bottomrule
  \end{tabularx}
  \caption{Initial rows and defining formulas for the seven families.}
  \label{tab:familySummary}
\end{table}

\Needspace{6\baselineskip}
The paper is organized as follows.  Section~\ref{sec:derangement} contains the
self-contained derangement argument and its larger deletion-stable family.
Sections~\ref{sec:cyclic} and~\ref{sec:evenTop} give two shorter applications,
and Sections~\ref{sec:superEulerian} and~\ref{sec:furtherFamilies} treat the
three further families. Section~\ref{sec:zigzag} proves the zig-zag result.
We conclude with acknowledgements.

\section{Derangement descent polynomials}
\label{sec:derangement}

We begin with the derangement part of the Fu--Lin--Zeng conjecture.  Our proof
reciprocates the descent polynomial and retains both constant completions of
the resulting Eulerian recurrence.  The finite symbol of a lowering operator
supplies the required orientation.
The corresponding descent triangle is indexed by OEIS \oeis{A219836}
\cite{OEISA219836}.

For \(n\geq1\), let \(\defin{\mathfrak S_n}\) denote the symmetric group.  If
\(\pi=\pi_1\dotsm\pi_n\in\mathfrak S_n\), we write \(\defin{\des(\pi)}\) and
\(\defin{\asc(\pi)}\) for the numbers of indices \(1\leq i<n\) for which
\(\pi_i>\pi_{i+1}\) and \(\pi_i<\pi_{i+1}\), respectively.

We write \(\defin{\mathfrak D_n}\subseteq\mathfrak S_n\) for the set of
derangements, that is, fixed-point-free permutations, and define
the \defin{derangement descent polynomial} by
\[
  \defin{D_n(t)}
  \coloneqq
  \sum_{\pi\in\mathfrak D_n} t^{\des(\pi)}.
\]
\Needspace{5\baselineskip}
Fu, Lin, and Zeng conjectured that these polynomials are real-rooted
\cite[Conjecture~5.2]{FuLinZeng2018}.  Theorem~\ref{thm:derangementRealRooted}
proves the derangement part of their conjecture.
Here descents are essential: this distribution is different from the
classical derangement polynomial defined using excedances.

\begin{theorem}
  \label{thm:derangementRealRooted}
  For every \(n\geq2\), the polynomial \(D_n(t)\) has only nonpositive
  real zeros.
\end{theorem}

Starting from descent-set formulas of J.~D\'esarm\'enien and D.~Foata and of
I.~M.~Gessel and C.~Reutenauer \cite{DesarmenienFoata1985,
GesselReutenauer1993}, Fu, Lin, and Zeng obtained the recurrence
\cite[Lemma~4.1]{FuLinZeng2018}
\begin{equation}
  D_n(t)
  =
  \bigl(1+(n-1)t\bigr)D_{n-1}(t)
  +t(1-t)D_{n-1}'(t)
  +(-1)^n t^{n-1}.
  \label{eq:fuLinZengRecurrence}
\end{equation}
For \(f\in\mathbb R[x]\), we define the \defin{Eulerian operator}
\[
  \defin{(T_nf)(x)}
  \coloneqq
  \bigl(1+(n-1)x\bigr)f(x)+x(1-x)f'(x).
\]
We also define
\[
  \defin{(U_nf)(x)}
  \coloneqq
  (n-1)xf(x)+x(1-x)f'(x).
\]
Thus the first two terms on the right of \eqref{eq:fuLinZengRecurrence} are
\((T_nD_{n-1})(t)\).

\subsection{Reciprocal form}
\label{sec:reciprocal}

Direct differentiation gives the intertwining identity
\[
  x^{n-1}(T_nf)(1/x)
  =T_n\bigl(x^{n-2}f(1/x)\bigr)
  \qquad(\deg f\leq n-2).
\]

We define the \defin{reciprocal derangement polynomial}
\begin{equation}
  \defin{Q_n(x)}
  \coloneqq
  x^{n-1}D_n(1/x).
  \label{eq:reciprocalPolynomial}
\end{equation}
Applying the intertwining identity to \eqref{eq:fuLinZengRecurrence} gives
\begin{equation}
  Q_n=T_nQ_{n-1}+(-1)^n.
  \label{eq:reciprocalRecurrence}
\end{equation}
The initial values are
\[
  Q_2=1,
  \qquad
  Q_3=2x.
\]
Equivalently,
\[
  Q_n(x)=\sum_{\pi\in\mathfrak D_n}x^{\asc(\pi)}.
\]
Consequently, \(Q_n\) has nonnegative coefficients, and
\(\deg Q_n=n-2\) for \(n\geq2\).

Reciprocation sends every nonzero root \(\rho\) to \(1/\rho\), while any
additional roots at zero are real.  It is therefore enough to prove that every
\(Q_n\) is real-rooted.

\subsection{Compatibility, interlacing, and finite symbols}
\label{sec:compatibility}

We regard the zero polynomial as real-rooted and set \(\deg 0=-\infty\).
A finite family \(f_1,\dotsc,f_s\) of real polynomials is
\defin{compatible} if
\[
  \lambda_1f_1+\dotsb+\lambda_sf_s
\]
is real-rooted for every \(\lambda_1,\dotsc,\lambda_s\geq0\).
For nonzero real-rooted polynomials \(g,f\) with positive leading
coefficients, we write
\[
  \defin{g\preceq f}
\]
if \(\deg f-\deg g\in\{0,1\}\) and their roots alternate in the following
orientation.  For constant \(g\), this condition is automatic.  Otherwise,
write \(\deg g=d\geq1\), list the roots of \(g\) as
\(u_1\geq\dotsb\geq u_d\) and those of \(f\) as \(v_i\), also in decreasing
order.  We require
\[
  \begin{aligned}
    v_1&\geq u_1\geq v_2\geq\dotsb\geq u_d\geq v_{d+1}
      &&\text{if }\deg f=d+1,\\
    v_1&\geq u_1\geq v_2\geq\dotsb\geq v_d\geq u_d
      &&\text{if }\deg f=d.
  \end{aligned}
\]
We call \(g\) a \defin{common interlacer} of \(f_1,\dotsc,f_s\) if
\(g\preceq f_i\) for every \(i\).
The univariate Hermite--Kakeya--Obreschkoff theorem says that every real
linear combination of two polynomials in this oriented relation is
real-rooted; see \cite[Section~1.2]{BorceaBranden2009}.  The compatibility
theorem of M.~Chudnovsky and P.~Seymour
\cite[Theorem~3.6]{ChudnovskySeymour2007} says that a common interlacer
implies compatibility and that, for polynomials with positive leading
coefficients, pairwise compatibility is equivalent to compatibility of the
whole finite family.

With this orientation, if the roots are simple and the interlacing is strict,
then the residues in the partial-fraction expansion of \(g/f\) at the roots
of \(f\) are positive.  One-sided compatibility alone does not determine
\(g\preceq f\).

\subsubsection{The finite-symbol criterion}
\label{subs:finiteSymbol}

A polynomial in several complex variables is \defin{stable} if it does not
vanish when all variables lie in the open upper half-plane.  We let \(L\) be a
real-linear operator on polynomials of degree at most \(d\), with input
variable \(x\).  For symbol variables \(y,z\), its \defin{finite algebraic
symbol} is
\[
  \defin{\Sigma_d(L)(y,z)}
  \coloneqq
  \sum_{k=0}^d\binom{d}{k}L(x^k)(y)z^{d-k}.
\]
We use the following univariate consequence of the finite-symbol theorem of
J.~Borcea and P.~Br\"and\'en \cite[Theorem~1.1]{BorceaBranden2009}.

\begin{theorem}
  \label{thm:finiteSymbolCriterion}
  If \(\Sigma_d(L)\) is stable, then \(L\) preserves real-rootedness on
  polynomials of degree at most \(d\).  Moreover, if
  \(g\preceq f\), \(\deg f,\deg g\leq d\), and \(Lf,Lg\) are nonzero with
  positive leading coefficients, then
  \[
    Lg\preceq Lf.
  \]
\end{theorem}

The oriented assertion follows directly from the same theorem.  The
Hermite--Biehler characterization says that \(g\preceq f\) is equivalent to
the stability of \(f(y)+zg(y)\).  Applying \(L\) in the variable \(y\) gives
the stable polynomial \((Lf)(y)+z(Lg)(y)\), and Hermite--Biehler then gives
\(Lg\preceq Lf\).

\subsection{Rolle factorizations for the Eulerian operators}
\label{sec:rolleFactorizations}

The zero-preserving properties of the Eulerian operators follow from the
factorizations
\begin{equation}
  T_nf
  =(1-x)^{n+1}
  \left(\frac{xf}{(1-x)^n}\right)'
  \label{eq:TnRolleFactorization}
\end{equation}
and
\begin{equation}
  U_nf
  =x(1-x)^n
  \left(\frac{f}{(1-x)^{n-1}}\right)'.
  \label{eq:UnRolleFactorization}
\end{equation}
On monomials they act by
\begin{align}
  T_n(x^k)
  &=(k+1)x^k+(n-1-k)x^{k+1},
  \label{eq:TnMonomial}\\
  U_n(x^k)
  &=kx^k+(n-1-k)x^{k+1}.
  \label{eq:UnMonomial}
\end{align}
In particular, both operators preserve nonnegative coefficients under the
degree bounds below.

\begin{lemma}
  \label{lem:eulerianOperatorPreservation}
  Let \(f\) have nonnegative coefficients and only nonpositive real zeros.
  Then the following statements hold.
  \begin{enumerate}
    \item If \(\deg f\leq n-2\), then \(T_nf\) has only nonpositive real
      zeros.
    \item If \(\deg f\leq n-2\), then \(U_nf\) has only nonpositive real
      zeros and, if \(f\ne0\),
      \begin{equation}
        f\preceq U_nf.
        \label{eq:UnInterlacing}
      \end{equation}
    \item If \(f_1,\dotsc,f_s\) are compatible, have nonnegative
      coefficients, and have degree at most \(n-2\), then both families
      \(T_nf_1,\dotsc,T_nf_s\) and \(U_nf_1,\dotsc,U_nf_s\) are compatible.
    \item If \(g\preceq f\), both polynomials satisfy the hypotheses above,
      and \(\deg f,\deg g\leq n-2\), then
      \begin{equation}
        T_ng\preceq T_nf.
        \label{eq:TnOrientedPreservation}
      \end{equation}
  \end{enumerate}
\end{lemma}

\begin{proof}
  Suppose first that the zeros of \(f\) are simple and strictly negative.
  Apply Rolle's theorem to the two rational functions in
  \eqref{eq:TnRolleFactorization} and \eqref{eq:UnRolleFactorization}.  Their
  limits at \(-\infty\), together with the zero at the origin in the second
  factorization, give the first two assertions and
  \eqref{eq:UnInterlacing}.  The constant case is immediate.  Multiple zeros
  and zeros at the origin follow by a root perturbation and a
  coefficientwise limit; \eqref{eq:TnMonomial} and \eqref{eq:UnMonomial}
  keep the relevant degrees fixed.

  If \(f_1,\dotsc,f_s\) are compatible and \(\lambda_i\geq0\), linearity
  allows us to apply the first two parts to
  \(\sum_i\lambda_if_i\).  This proves componentwise preservation for both
  operators.

  It remains to prove the oriented assertion.  We put \(d=n-2\).  The case
  \(d=0\) is immediate, so suppose that \(d\geq1\).  Equation
  \eqref{eq:TnMonomial} gives
  \[
    \Sigma_d(T_n)(y,z)
    =(y+z)^{d-1}
    \bigl(y^2+(d+1)yz+(d+1)y+z\bigr).
  \]
  This symbol is stable.  Indeed, \(y+z\) is stable, and a zero of the second
  factor would satisfy
  \begin{equation}
    z=-\frac{y(y+d+1)}{(d+1)y+1}.
    \label{eq:TnSymbolRoot}
  \end{equation}
  Writing \(y=\sigma+i\tau\), where \(\tau>0\), we find
  \[
    \begin{aligned}
      \operatorname{Im}\frac{y(y+d+1)}{(d+1)y+1}
      &=\frac{\tau}{|(d+1)y+1|^2}\\[-2pt]
      &\quad{}\times\left(
          (d+1)\left(\left(\sigma+\frac1{d+1}\right)^2+\tau^2\right)
          {}+\frac{(d+1)^2-1}{d+1}
        \right)>0.
    \end{aligned}
  \]
  Thus \eqref{eq:TnSymbolRoot} forces \(\operatorname{Im}z<0\).  Theorem
  \ref{thm:finiteSymbolCriterion} proves
  \eqref{eq:TnOrientedPreservation}; equation~\eqref{eq:TnMonomial} also
  guarantees nonzero outputs with positive leading coefficients.  This
  completes the proof.
\end{proof}

\subsection{The mixed \texorpdfstring{\(T/U\)}{T/U} lemma}
\label{sec:mixedLemma}

We require the following mixed compatibility statement.

\begin{lemma}
  \label{lem:mixedRolle}
  Let \(f,g\) have nonnegative coefficients and only nonpositive real zeros.
  Suppose that
  \[
    g\preceq f,
    \qquad
    \deg f=\deg g+1\leq n-1.
  \]
  Then \(T_ng\) and \(U_{n+1}f\) are compatible.
\end{lemma}

\begin{proof}
  Fix \(\lambda>0\).  We prove that
  \[
    \defin{\mathcal P_\lambda}
    \coloneqq
    U_{n+1}f+\lambda T_ng
  \]
  is real-rooted.  Scaling then gives every combination with two positive
  coefficients, while Lemma~\ref{lem:eulerianOperatorPreservation} gives the
  endpoint combinations.

  Put \(\defin{h}\coloneqq f+\lambda g\) and
  \(d=\deg f=\deg h\); then \(\deg g=d-1\).
  Since \(g\preceq f\), the polynomial \(h\) is real-rooted and
  \(g\preceq h\).  The identity \(T_ng-U_{n+1}g=(1-x)g\) rewrites the pencil as
  \[
    \mathcal P_\lambda=U_{n+1}h+\lambda(1-x)g.
  \]

  We first assume that the zeros
  \[
    \xi_1>\xi_2>\dotsb>\xi_d
  \]
  of \(h\) are simple and strictly negative and that \(g\preceq h\) is strict.
  We then have
  \begin{equation}
    \frac{g(x)}{h(x)}
    =\sum_{i=1}^d\frac{c_i}{x-\xi_i},
    \qquad
    c_i>0.
    \label{eq:mixedPartialFractions}
  \end{equation}
  There is no polynomial part because \(\deg g=d-1\).

  At \(x=0\), equation~\eqref{eq:mixedPartialFractions} gives
  \[
    \sum_{i=1}^d\frac{\lambda c_i}{-\xi_i}
    =\frac{\lambda g(0)}{f(0)+\lambda g(0)}
    \leq1.
  \]
  The denominator is positive because all roots of \(h\) are strictly
  negative.  If \(f(0)>0\), the inequality is strict.  If \(f(0)=0\) and
  \(d\geq2\), the sum equals \(1\), but it contains at least two positive
  terms.  In either case, each summand is strictly smaller than \(1\), and
  therefore
  \begin{equation}
    \xi_i+\lambda c_i<0.
    \label{eq:mixedShiftedResidueNegative}
  \end{equation}

  It remains to treat the boundary \(d=1\) and \(f(0)=0\).  We then have
  \(f(x)=\alpha x\) and \(g(x)=\beta\) for some \(\alpha,\beta>0\), and
  \[
    \mathcal P_\lambda(x)
    =(\alpha x+\lambda\beta)\bigl(1+(n-1)x\bigr).
  \]
  Thus the boundary pencil is real-rooted with nonpositive zeros.

  Outside this boundary case, evaluation at a zero \(\xi_i\) of \(h\) gives
  \begin{align}
    \mathcal P_\lambda(\xi_i)
    &=\xi_i(1-\xi_i)h'(\xi_i)
      +\lambda(1-\xi_i)g(\xi_i)\notag\\
    &=(1-\xi_i)h'(\xi_i)
      \left(
        \xi_i+\lambda\frac{g(\xi_i)}{h'(\xi_i)}
      \right)\notag\\
    &=(1-\xi_i)h'(\xi_i)(\xi_i+\lambda c_i).
    \label{eq:mixedSignEvaluation}
  \end{align}
  By \eqref{eq:mixedShiftedResidueNegative}, these values alternate with the
  opposite signs from \(h'(\xi_i)\).  At the right endpoint,
  \begin{equation}
    \mathcal P_\lambda(0)=\lambda g(0)>0.
    \label{eq:mixedValueAtZero}
  \end{equation}
  The degree of \(\mathcal P_\lambda\) is \(d+1\), and its leading
  coefficient is
  \(n-d\) times that of \(f\), hence positive.  Its sign at \(-\infty\),
  together with
  \eqref{eq:mixedSignEvaluation}--\eqref{eq:mixedValueAtZero}, gives one zero
  in each interval
  \[
    (\xi_1,0),\quad
    (\xi_2,\xi_1),\quad\dotsc,\quad
    (\xi_d,\xi_{d-1}),\quad
    (-\infty,\xi_d).
  \]
  These are \(d+1\) distinct negative real zeros, and hence all zeros of
  \(\mathcal P_\lambda\) are real.

  The non-strict cases follow by simultaneously perturbing the two
  interlacing root lists to a strictly interlacing pair with simple negative
  zeros.  The degrees remain fixed, so coefficientwise closure of
  real-rootedness gives the result on taking the limit.  This completes the
  proof.
\end{proof}

\subsection{The crossed-completion lemma}
\label{sec:crossedCompletion}

We handle the alternating constant in \eqref{eq:reciprocalRecurrence} by
retaining both possible completions \(h\) and \(h+1\).  The following lemma
requires an oriented common interlacer, not merely pairwise compatibility.

\begin{lemma}
  \label{lem:crossedCompletion}
  Suppose that \(h\) and \(h+1\) have nonnegative coefficients and only
  nonpositive real zeros, and suppose that \(\deg h\leq n-2\).  Then
  \begin{equation}
    T_nh,
    \qquad
    T_n(h+1),
    \qquad
    U_{n+1}h,
    \qquad
    U_{n+1}(h+1)
    \label{eq:crossedCompletionOutputs}
  \end{equation}
  have a common interlacer after zero members are omitted. In particular,
  every pair among them is
  compatible.
\end{lemma}

\begin{proof}
  We first treat \(d=\deg h\geq2\).  Since \((h+1)'=h'\), Rolle's theorem
  and the oriented part of Lemma~\ref{lem:eulerianOperatorPreservation} give
  \[
    h'\preceq h,
    \qquad
    h'\preceq h+1,
    \qquad
    T_nh'\preceq T_nh,
    \qquad
    T_nh'\preceq T_n(h+1).
  \]
  We prove the same orientation for the two \(U\)-outputs; the
  compatibility in Lemma~\ref{lem:mixedRolle} does not determine it.

  We define the \defin{lowering operator}
  \[
    \defin{L_np}\coloneqq np+(1-x)p'.
  \]
  On polynomials of degree at most \(d\), its finite algebraic symbol is
  \[
    \Sigma_d(L_n)(y,z)
    =(y+z)^{d-1}\bigl((n-d)y+nz+d\bigr).
  \]
  Since \(d\leq n-2\), both factors are stable.  Thus
  Theorem~\ref{thm:finiteSymbolCriterion} shows that \(L_n\) preserves
  oriented interlacing on polynomials of degree at most \(d\).

  Since all zeros of \(h\) are nonpositive, Rolle's theorem, followed by
  adjoining the zero root of \(xh'\), gives \(h\preceq xh'\).  The monomial
  action of \(L_n\) preserves nonnegative coefficients and positive leading
  coefficients in this degree range.  Hence
  \[
    L_nh\preceq L_n(xh').
  \]

  Adjoining a zero root to the lower member of a same-degree oriented pair
  reverses the placement.
  Using \(L_n(xh')=T_nh'\) and \(xL_nh=U_{n+1}h\), we obtain
  \[
    T_nh'\preceq U_{n+1}h.
  \]
  Applying the same argument to \(h+1\), which has the same derivative and
  degree, gives \(T_nh'\preceq U_{n+1}(h+1)\).  Hence \(T_nh'\) is a common
  interlacer of all four polynomials in \eqref{eq:crossedCompletionOutputs}.

  We finish with the low-degree cases.  If \(h\) is constant, a nonzero
  constant is a common interlacer of all the nonzero outputs.
  If \(h(x)=c_0+c_1x\) has degree one, then \(c_1>0\),
  \(n\geq3\), and \(T_nh'=c_1(1+(n-1)x)\) has root
  \(s=-1/(n-1)\).  Direct evaluation gives
  \[
    \begin{aligned}
      T_nh(s)=T_n(h+1)(s)&=c_1s(1-s)<0,\\
      U_{n+1}h(s)&=nc_0s\leq0,
      &U_{n+1}(h+1)(s)&=n(c_0+1)s<0.
    \end{aligned}
  \]
  All four quadratic outputs have positive leading coefficient, so
  these evaluations place \(s\) between their two roots.  Thus \(T_nh'\) is
  a common interlacer, completing the proof in every degree.
\end{proof}

\subsection{A larger deletion-stable family}
\label{sec:largerFamily}

We now place the reciprocal derangement polynomials in a larger family that
is preserved by the deletion recurrence.
We fix an integer \(a\geq1\), set \(\defin{Q_0^{(a)}}\coloneqq1\), and, for
\(r\geq1\), define
\begin{equation}
  \defin{Q_r^{(a)}(x)}
  \coloneqq T_{a+r+1}Q_{r-1}^{(a)}+(-1)^r.
  \label{eq:largerRecurrence}
\end{equation}
The first values are
\begin{equation}
  Q_0^{(a)}=1,
  \qquad
  Q_1^{(a)}=(a+1)x,
  \qquad
  Q_2^{(a)}=(1+(a+1)x)^2.
  \label{eq:largerInitialValues}
\end{equation}
Since \(T_nf=f+U_nf\), recurrence \eqref{eq:largerRecurrence} one step
earlier gives the two-summand decomposition
\begin{equation}
  Q_r^{(a)}
  =T_{a+r}Q_{r-2}^{(a)}+U_{a+r+1}Q_{r-1}^{(a)}
  \qquad(r\geq2).
  \label{eq:largerDecomposition}
\end{equation}

Starting from \eqref{eq:largerInitialValues}, the monomial formulas
\eqref{eq:TnMonomial}--\eqref{eq:UnMonomial} and
\eqref{eq:largerDecomposition} give, for every \(r\geq2\),
\[
  \deg Q_r^{(a)}=r,
  \qquad
  \deg\bigl(T_{a+r}Q_{r-2}^{(a)}\bigr)=r-1,
  \qquad
  \deg\bigl(U_{a+r+1}Q_{r-1}^{(a)}\bigr)=r.
\]
All these polynomials have nonnegative coefficients and positive leading
coefficients.

\begin{theorem}
  \label{thm:largerFamily}
  For every \(a\geq1\) and every \(r\geq2\), the pair
  \[
    \left(
      T_{a+r}Q_{r-2}^{(a)},
      U_{a+r+1}Q_{r-1}^{(a)}
    \right)
  \]
  is compatible.  Consequently, every \(Q_r^{(a)}\) has only nonpositive
  real zeros.
\end{theorem}

\begin{proof}
  We proceed by induction on \(r\).  At \(r=2\),
  \[
    T_{a+2}Q_0^{(a)}=1+(a+1)x,
    \qquad
    U_{a+3}Q_1^{(a)}=(a+1)x\bigl(1+(a+1)x\bigr).
  \]
  The second is \((a+1)x\) times the first, so the two polynomials interlace
  and are compatible.  Their sum is \(Q_2^{(a)}=(1+(a+1)x)^2\).

  Suppose the assertion has been proved through index \(r\).  The initial
  values and \eqref{eq:largerDecomposition} show that all preceding
  \(Q_s^{(a)}\) are then real-rooted.  We show that the three polynomials
  \[
    \begin{aligned}
      \textup{(i)}\quad &T_{a+r+1}Q_{r-1}^{(a)},\\
      \textup{(ii)}\quad &U_{a+r+2}T_{a+r}Q_{r-2}^{(a)},\\
      \textup{(iii)}\quad &U_{a+r+2}U_{a+r+1}Q_{r-1}^{(a)}
    \end{aligned}
  \]
  are pairwise compatible.

  First, (ii) and (iii) arise by applying \(U_{a+r+2}\) to the compatible pair
  at index \(r\).  They are compatible by
  Lemma~\ref{lem:eulerianOperatorPreservation}.

  Next, recurrence \eqref{eq:largerRecurrence} shows that
  \(T_{a+r}Q_{r-2}^{(a)}\) and \(Q_{r-1}^{(a)}\) differ by \(1\).  Both have
  nonnegative coefficients and, by induction, only nonpositive real zeros.
  Lemma~\ref{lem:crossedCompletion}, with parameter \(a+r+1\), therefore gives
  compatibility of (i) and (ii).

  Finally, Lemma~\ref{lem:eulerianOperatorPreservation} gives
  \[
    Q_{r-1}^{(a)}
    \preceq
    U_{a+r+1}Q_{r-1}^{(a)}.
  \]
  Applying Lemma~\ref{lem:mixedRolle} with parameter \(a+r+1\) gives
  compatibility of (i) and (iii).

  We have proved pairwise compatibility of the three polynomials.  Their
  leading coefficients are positive, so the Chudnovsky--Seymour theorem
  implies that the triple is compatible.  At index \(r+1\), the first member
  of the asserted pair is (i).  By linearity and
  \eqref{eq:largerDecomposition}, its second member is (ii) plus (iii).
  Hence the pair at index \(r+1\) is compatible.  Its sum is
  \(Q_{r+1}^{(a)}\), which is therefore real-rooted.  This completes the
  induction.
\end{proof}

\subsection{Specialization to derangements}
\label{sec:specialization}

We take \(a=1\) and set \(r=n-2\).  For \(n\geq3\), recurrence
\eqref{eq:largerRecurrence} becomes
\[
  Q_{n-2}^{(1)}
  =T_nQ_{n-3}^{(1)}+(-1)^{n-2}
  =T_nQ_{n-3}^{(1)}+(-1)^n.
\]
Since \(Q_0^{(1)}=1=Q_2\), comparison with
\eqref{eq:reciprocalRecurrence} gives, by induction,
\[
  Q_{n-2}^{(1)}=Q_n
  \qquad(n\geq2).
\]
Theorem~\ref{thm:largerFamily} applies when \(n\geq4\), while the initial
values \(Q_2=1\) and \(Q_3=2x\) cover the remaining cases.  Thus every
reciprocal derangement polynomial \(Q_n\) has only nonpositive real zeros.
Equation~\eqref{eq:reciprocalPolynomial} now shows that every \(D_n(t)\) has
only nonpositive real zeros.  This proves
Theorem~\ref{thm:derangementRealRooted}.

\subsection{Permutation interpretation of the larger family}
\label{sec:permutationModel}

The family in Subsection~\ref{sec:largerFamily} has a natural deletion-stable
permutation interpretation.  For \(\pi\in\mathfrak S_N\), define
\[
  \defin{\ell(\pi)}
  \coloneqq
  \max\bigl\{1\leq m\leq N:
    \pi^{-1}(1)>\pi^{-1}(2)>\dotsb>\pi^{-1}(m)
  \bigr\}.
\]
For \(a\geq1\), \(r\geq0\), and \(N=a+r+1\), we have
\begin{equation}
  Q_r^{(a)}(x)
  =
  \sum_{\substack{\pi\in\mathfrak S_N\\
      \ell(\pi)\geq a+1\\
      \ell(\pi)\equiv a+1\!\!\pmod 2}}
    x^{\asc(\pi)}.
  \label{eq:permutationInterpretation}
\end{equation}
We prove the identity by showing that its right side satisfies the recurrence
for \(Q_r^{(a)}\).  At \(r=0\), only the decreasing permutation occurs, so
the right side equals \(1\).

We insert the new maximum into a permutation of size \(N-1\).  If the old
permutation has \(k\) ascents, then the insertion positions contribute
\[
  (k+1)x^k+(N-1-k)x^{k+1}=T_N(x^k).
\]
For every old permutation except the decreasing one, insertion leaves
\(\ell\) unchanged.  For the decreasing permutation, insertion of the new
maximum in the first position changes \(\ell=N-1\) to \(\ell=N\).  If \(r\) is
even, the old decreasing permutation is excluded and the new decreasing
permutation is included, giving \(+1\).  If \(r\) is odd, the old decreasing
permutation is included and the new decreasing permutation is excluded,
giving \(-1\).

It follows that the right side of
\eqref{eq:permutationInterpretation} satisfies \eqref{eq:largerRecurrence},
which proves the identity.  Theorem~\ref{thm:largerFamily}, together with the
initial cases \(r=0,1\), now shows that this ascent polynomial has only
nonpositive real zeros.  For \(a=1\), it is the reciprocal derangement
polynomial \(Q_{r+2}\).

\section{Cyclic descents of northeast paths}
\label{sec:cyclic}

We identify a northeast path from \((0,0)\) to \((n,n)\) with its binary step
word \(w_1\dotsm w_{2n}\), using \(0\) for an east step and \(1\) for a
north step.  With indices taken modulo \(2n\), a \defin{cyclic descent} is
an index \(i\) for which \(w_i=1\) and \(w_{i+1}=0\).  We let
\(\defin{\operatorname{cdes}(w)}\) denote the number of cyclic descents of
\(w\), and define
\[
  \defin{C_n(t)}
  \coloneqq
  \sum_{\substack{w\in\{0,1\}^{2n}\\
      w\text{ has \(n\) zeros and \(n\) ones}}}
  t^{\operatorname{cdes}(w)}.
\]
The coefficient formula of P.~Alexandersson, S.~Linusson, S.~Potka, and
J.~Uhlin
\cite[Lemma~43]{AlexanderssonLinussonPotkaUhlin2021} gives
\begin{equation}
  C_n(t)
  =2\sum_{k=1}^{n}\binom nk\binom{n-1}{k-1}t^k.
  \label{eq:cyclicCoefficientFormula}
\end{equation}
OEIS \oeis{A335340} asks whether these polynomials are real-rooted
\cite{OEISA335340}.  We prove the stronger statement below.

\begin{theorem}
  \label{thm:cyclicDescents}
  For every \(n\geq1\), the polynomial \(C_n(t)\) has one simple zero at the
  origin and \(n-1\) simple negative zeros.
\end{theorem}

\begin{proof}
  We use the type~\(B\) Narayana polynomial
  \[
    \defin{N_n^{B}(t)}\coloneqq\sum_{k=0}^{n}\binom nk^2t^k.
  \]
  Since
  \(\binom{n-1}{k-1}=\frac{k}{n}\binom nk\), equation
  \eqref{eq:cyclicCoefficientFormula} becomes
  \begin{equation}
    C_n(t)=\frac{2t}{n}(N_n^{B})'(t).
    \label{eq:cyclicDerivative}
  \end{equation}
  We use the Legendre specialization \(\mathsf P_n^{(0,0)}\) of the Jacobi
  polynomials.  The standard M\"obius transformation gives
  \[
    N_n^{B}(t)
    =(1-t)^n
    \mathsf P_n^{(0,0)}\left(\frac{1+t}{1-t}\right).
  \]
  The polynomial \(\mathsf P_n^{(0,0)}\) has \(n\) simple zeros in
  \((-1,1)\).  The inverse of the displayed fractional-linear map sends this
  interval onto \((-\infty,0)\), so \(N_n^{B}\) has \(n\) simple negative
  zeros.  Rolle's theorem gives \(n-1\) simple negative zeros of
  \((N_n^{B})'\).  Finally, \((N_n^{B})'(0)=n^2\), so
  equation~\eqref{eq:cyclicDerivative} adds a simple zero at the origin.
  This completes the proof.
\end{proof}

\section{Descents whose top is even}
\label{sec:evenTop}

For \(\pi=\pi_1\dotsc\pi_N\in\mathfrak S_N\), a
\defin{descent with even top} is an index \(i\) such that
\(\pi_i>\pi_{i+1}\) and \(\pi_i\) is even.  We define
\[
  \defin{R_N(t)}
  \coloneqq
  \sum_{\pi\in\mathfrak S_N}
  t^{\#\{\text{descents of \(\pi\) with even top}\}}.
\]
The coefficient triangle is indexed by OEIS \oeis{A134434}
\cite{OEISA134434}.
S.~Kitaev and J.~Remmel proved the differential recurrences
\cite[Theorem~1 and Corollary~1]{KitaevRemmel2007}
\begin{align}
  R_{2n+1}(t)
  &=(1-t)R_{2n}'(t)+(2n+1)R_{2n}(t),
  \label{eq:evenToOdd}\\
  R_{2n+2}(t)
  &=t(1-t)R_{2n+1}'(t)
    +\bigl(1+(2n+1)t\bigr)R_{2n+1}(t),
  \label{eq:oddToEven}
\end{align}
for \(n\geq1\), with \(R_2(t)=1+t\).  J.~T.~Hall and J.~B.~Remmel
identified the more general prescribed-top and prescribed-bottom descent
distributions with hit polynomials of Ferrers boards, up to rook equivalence
\cite[Proposition~4.3]{HallRemmel2008}.  Together with the real-rootedness
theorem of J.~Haglund, K.~Ono, and D.~G.~Wagner
\cite[Theorem~1]{HaglundOnoWagner1999}, this already gives rowwise
real-rootedness.  The recurrences above also give strict interlacing of
consecutive rows.

\begin{theorem}
  \label{thm:evenTopInterlacing}
  For every \(N\geq2\), the polynomial \(R_N(t)\) has only simple negative
  zeros.  Moreover, the zeros of \(R_N(t)\) and \(R_{N+1}(t)\) strictly
  interlace.
\end{theorem}

\begin{proof}
  The recurrences show inductively that
  \(\deg R_{2n}=\deg R_{2n+1}=n\) and that every leading coefficient is
  positive: either transition multiplies the preceding leading coefficient
  by \(n+1\).  We proceed by induction, beginning with
  \(R_2(t)=1+t\).  We first assume
  that the zeros of \(R_{2n}\) are
  \[
    q_1<q_2<\dotsb<q_n<0.
  \]
  The leading coefficient is positive, and hence \(R_{2n}'(q_i)\) has sign
  \((-1)^{n-i}\).  Evaluating
  \eqref{eq:evenToOdd} at \(q_i\) gives
  \[
    R_{2n+1}(q_i)=(1-q_i)R_{2n}'(q_i).
  \]
  Since \(1-q_i>0\), these values alternate in sign.  The sign of
  \(R_{2n+1}\) at minus infinity is \((-1)^n\), so it has exactly one zero
  in each of
  \[
    (-\infty,q_1),(q_1,q_2),\dotsc,(q_{n-1},q_n).
  \]
  Writing the zeros as \(p_1<\dotsb<p_n\), we obtain
  \begin{equation}
    p_1<q_1<p_2<q_2<\dotsb<p_n<q_n<0.
    \label{eq:evenOddInterlacing}
  \end{equation}

  Next evaluate \eqref{eq:oddToEven} at the zeros \(p_i\):
  \[
    R_{2n+2}(p_i)=p_i(1-p_i)R_{2n+1}'(p_i).
  \]
  Here \(p_i(1-p_i)<0\), so the signs again alternate, with the derivative
  signs reversed.  Since \(R_{2n+2}(0)=R_{2n+1}(0)>0\), the polynomial
  \(R_{2n+2}\) has one zero in each of
  \[
    (-\infty,p_1),(p_1,p_2),\dotsc,(p_{n-1},p_n),(p_n,0).
  \]
  If these zeros are \(s_1<\dotsb<s_{n+1}\), then
  \begin{equation}
    s_1<p_1<s_2<p_2<\dotsb<s_n<p_n<s_{n+1}<0.
    \label{eq:oddEvenInterlacing}
  \end{equation}
  Equations \eqref{eq:evenOddInterlacing} and
  \eqref{eq:oddEvenInterlacing} close the induction.  Every displayed
  inequality is strict, so all zeros are simple.  This proves the theorem.
\end{proof}

\section{Super-Eulerian polynomials}
\label{sec:superEulerian}

U.~Shankar introduced the \((l,r)\)-Eulerian numbers, a two-parameter family
in which the two coefficients of the Eulerian recurrence are raised to a
fixed positive integral power~\cite{Shankar2025}.  We consider the case
\(r=1\).  Fix \(l\geq1\), put \(\defin{E_{1,0}^{(l)}}\coloneqq1\), and, for
\(n\geq2\), define
\begin{equation}
  \defin{E_{n,k}^{(l)}}
  \coloneqq
  (k+1)^lE_{n-1,k}^{(l)}+(n-k)^lE_{n-1,k-1}^{(l)},
  \label{eq:superEulerianRecurrence}
\end{equation}
where entries outside \(0\leq k\leq n-1\) are zero.  The corresponding
\defin{super-Eulerian polynomial} is
\[
  \defin{E_n^{(l)}(t)}
  \coloneqq\sum_{k=0}^{n-1}E_{n,k}^{(l)}t^k.
\]
For \(l=1\), these are the ordinary Eulerian polynomials.  Their simple
negative zeros go back to Frobenius~\cite{Frobenius1910}, while their
gamma-positive expansion is classical work of Foata and
Sch\"utzenberger~\cite{FoataSchutzenberger1970}.  Br\"and\'en's permutation
actions and the surveys by C.~A.~Athanasiadis and T.~K.~Petersen give modern
accounts and extensions of the latter result
\cite{Branden2008,Athanasiadis2018,Petersen2015}.  These classical results
settle the first power only.  For the deformation, Shankar proved that every
row is palindromic and log-concave and that the rows are gamma-positive for
\(l=2\).  Real-rootedness for \(l\geq2\), and gamma-positivity for \(l\geq3\),
are recorded there only as observations from limited
data~\cite[Section~4.2]{Shankar2025}.
Since \(E_n^{(l)}\) is palindromic of degree \(n-1\), it has a unique
expansion
\[
  E_n^{(l)}(t)
  =\sum_{i=0}^{\lfloor(n-1)/2\rfloor}
    \gamma_{n,i}^{(l)}t^i(1+t)^{n-1-2i}.
\]
We call \(\sum_i\gamma_{n,i}^{(l)}t^i\) its
\defin{gamma-polynomial}.

Let \(\defin{\mathcal{PF}}\) be the family of real polynomials with
nonnegative coefficients and only nonpositive real zeros.  For
\(f=\sum_i a_it^i\) and \(g=\sum_i b_it^i\), let
\[
  \defin{f\had g}\coloneqq\sum_i a_ib_it^i
\]
be their coefficientwise Hadamard product.  The P\'olya--Schur--Wagner
theorem says that \(\mathcal{PF}\) is closed under Hadamard
products~\cite{SchurPolya1914,Wagner1992}.  The two-pair theorem of
J.~Garloff and D.~G.~Wagner says that, for \(f,g,p,r\in\mathcal{PF}\) such
that the two products below are nonzero,
\begin{equation}
  f\preceq g\quad\text{and}\quad p\preceq r
  \quad\Longrightarrow\quad
  f\had p\preceq g\had r.
  \label{eq:garloffWagnerTwoPair}
\end{equation}
See~\cite[Theorem~4(b)]{GarloffWagner1996}.  We also use the cone property
\cite[Lemma~3.2(ii)]{AlexanderssonNabawanda2022}: if \(h\preceq f\) and
\(h\preceq g\), then \(h\preceq af+bg\) for \(a,b\geq0\) not both zero.

\begin{theorem}
  \label{thm:superEulerian}
  Let \(l\geq1\) be an integer. For every \(n\geq1\), the polynomial
  \(E_n^{(l)}(t)\) has only nonpositive real zeros.  Moreover,
  \[
    E_n^{(l)}\preceq E_{n+1}^{(l)}.
  \]
  The gamma-polynomial of \(E_n^{(l)}\) is real-rooted and has nonnegative
  coefficients.
\end{theorem}

\begin{proof}
  We first prove real-rootedness.  Write \(\theta f=tf'\), and normalize by
  \[
    \defin{\widetilde E_n^{(l)}(t)}
    \coloneqq
    \sum_{i=0}^{n-1}
      \frac{E_{n,i}^{(l)}}{\binom{n-1}i^l}t^i
    \qquad(n\geq1).
  \]
  A coefficient comparison in \eqref{eq:superEulerianRecurrence} gives,
  for \(n\geq2\),
  \begin{equation}
    \widetilde E_n^{(l)}
    =(1+t)
      \left(
        \frac{(\theta+1)(n-1-\theta)}{n-1}
      \right)^l
      \widetilde E_{n-1}^{(l)}.
    \label{eq:superEulerianNormalizedTransfer}
  \end{equation}
  Both operators in this formula preserve \(\mathcal{PF}\) on polynomials of
  degree at most \(n-2\).  Indeed, \((\theta+1)f=(tf)'\), so one case follows
  from Rolle's theorem.  Applying the same argument to
  \(t^{n-1}f(1/t)\), and then reciprocating back, shows that
  \((n-1-\theta)f\in\mathcal{PF}\).  Starting with
  \(\widetilde E_1^{(l)}=1\),
  equation~\eqref{eq:superEulerianNormalizedTransfer} therefore shows that
  every normalized row belongs to \(\mathcal{PF}\).

  For the Hadamard factorization, set
  \[
    \defin{B_n^{(l)}(t)}
    \coloneqq\sum_{i=0}^{n-1}\binom{n-1}i^lt^i.
  \]
  We have the coefficientwise factorization
  \begin{equation}
    E_n^{(l)}=B_n^{(l)}\had\widetilde E_n^{(l)}.
    \label{eq:superEulerianHadamardFactorization}
  \end{equation}
  The polynomial \(B_n^{(l)}\) is the \(l\)-fold Hadamard power of
  \((1+t)^{n-1}\), so \(B_n^{(l)}\in\mathcal{PF}\).  Closure under Hadamard
  products in~\eqref{eq:superEulerianHadamardFactorization} now shows that
  \(E_n^{(l)}\in\mathcal{PF}\).

  By Shankar's palindromicity~\cite[Section~4.2]{Shankar2025},
  \(E_n^{(l)}\) is palindromic of degree \(n-1\).  Since
  \(\binom{n-1}i=\binom{n-1}{n-1-i}\), so is
  \(\widetilde E_n^{(l)}\).  For \(n\geq2\), we have
  \[
    (1+t)^{n-1}
      \preceq (1+t)^{n-2}(1+nt),
    \qquad
    (1+t)^{n-1}
      \preceq t(1+t)^{n-2}(t+n).
  \]
  For the interlacing step, set
  \[
    \defin{\Lambda_n^{(l)}(t)}
      \coloneqq\sum_{i=0}^{n-1}(i+1)^l\binom{n-1}i^lt^i.
  \]
  Thus \(B_n^{(l)}\), \(\Lambda_n^{(l)}\), and
  \(t^n\Lambda_n^{(l)}(1/t)\) are the corresponding \(l\)-fold Hadamard
  powers, respectively.  Iterating \eqref{eq:garloffWagnerTwoPair} gives
  \begin{equation}
    B_n^{(l)}\preceq\Lambda_n^{(l)},
    \qquad
    B_n^{(l)}\preceq t^n\Lambda_n^{(l)}(1/t).
    \label{eq:superEulerianKernelRelations}
  \end{equation}

  Palindromicity and equation~\eqref{eq:superEulerianRecurrence} give
  \[
    E_{n+1}^{(l)}
    =(\theta+1)^lE_n^{(l)}
     +t^n\bigl((\theta+1)^lE_n^{(l)}\bigr)(1/t).
  \]
  In addition to~\eqref{eq:superEulerianHadamardFactorization}, we have
  \[
    (\theta+1)^lE_n^{(l)}
      =\Lambda_n^{(l)}\had\widetilde E_n^{(l)},
    \qquad
    t^n\bigl((\theta+1)^lE_n^{(l)}\bigr)(1/t)
      =t^n\Lambda_n^{(l)}(1/t)\had t\widetilde E_n^{(l)}.
  \]
  Since \(\widetilde E_n^{(l)}\preceq\widetilde E_n^{(l)}\) and
  \(\widetilde E_n^{(l)}\preceq t\widetilde E_n^{(l)}\), the two relations in
  \eqref{eq:superEulerianKernelRelations} and
  \eqref{eq:garloffWagnerTwoPair} give
  \[
    E_n^{(l)}\preceq(\theta+1)^lE_n^{(l)},
    \qquad
    E_n^{(l)}
      \preceq t^n\bigl((\theta+1)^lE_n^{(l)}\bigr)(1/t).
  \]
  The cone property now gives \(E_n^{(l)}\preceq E_{n+1}^{(l)}\).
  The case \(n=1\) is \(1\preceq1+t\).

  Finally, \(E_n^{(l)}\) is palindromic, has nonnegative coefficients, and is
  real-rooted.  The standard gamma-polynomial equivalence for palindromic
  polynomials therefore shows that its gamma-polynomial is real-rooted and
  has nonnegative coefficients~\cite[Observation~4.2]{Petersen2015}.  This
  completes the proof.
\end{proof}

\subsection{Interlacing for restricted subset sums}
\label{subs:superEulerianPrefixes}

We first express each row as a weighted sum over subsets.  For
\(B\subseteq[n]\) with \(1\in B\), put
\[
  \defin{r_i(B)}\coloneqq |B\cap[i-1]|,
  \qquad
  \defin{c_n(B)}\coloneqq
    \prod_{i\in B}\bigl(i-r_i(B)\bigr)
    \prod_{i\in[n]\setminus B}r_i(B).
\]
Then
\begin{equation}
  E_n^{(l)}(t)
  =\sum_{\substack{B\subseteq[n]\\1\in B}}
    c_n(B)^l t^{|B|-1}.
  \label{eq:superEulerianSubsetSum}
\end{equation}
Indeed, the sum is \(1\) at \(n=1\).  A set of size \(k+1\) not containing
\(n\) gains a factor \((k+1)^l\) from its final position.  A set of size
\(k+1\) containing \(n\) gains a factor \((n-k)^l\), and removing \(n\)
leaves a set of size \(k\).  Thus the sums satisfy
\eqref{eq:superEulerianRecurrence}.

For \(1\leq m\leq n\), we restrict the sum to subsets contained in \([m]\):
\begin{equation}
  \defin{P_{n,m}^{(l)}(t)}
  \coloneqq\sum_{\substack{B\subseteq[m]\\1\in B}}
    c_n(B)^l t^{|B|-1}.
  \label{eq:superEulerianPrefixDefinition}
\end{equation}
The endpoints are \(P_{n,1}^{(l)}=1\) and
\(P_{n,n}^{(l)}=E_n^{(l)}\).  Each position after \(m\) contributes a
factor \(|B|\), so \(c_n(B)=|B|^{n-m}c_m(B)\).  Consequently,
\begin{equation}
  P_{n,m}^{(l)}=(\theta+1)^{l(n-m)}E_m^{(l)},
  \qquad \theta=t\frac{d}{dt}.
  \label{eq:superEulerianPrefixOperator}
\end{equation}
For example, at \(n=3\) these polynomials are
\[
  P_{3,1}^{(l)}=1,\qquad
  P_{3,2}^{(l)}=1+2^lt,\qquad
  P_{3,3}^{(l)}=1+2^{l+1}t+t^2.
\]

\begin{proposition}
  \label{prop:superEulerianPrefixes}
  Let \(l\geq1\) be an integer.  For \(n\geq1\) and \(1\leq m<n\),
  \[
    P_{n,m}^{(l)}\preceq P_{n,m+1}^{(l)}.
  \]
  In particular, every polynomial \(P_{n,m}^{(l)}\) has only nonpositive
  real zeros.
\end{proposition}

\begin{proof}
  We prove the stronger comparison
  \[
    (\theta+1)^lE_m^{(l)}\preceq E_{m+1}^{(l)}
  \]
  and then apply a common power of \(\theta+1\).
  For \(m\geq2\), the elementary relation
  \[
    (1+t)^{m-2}(1+mt)
      \preceq t(1+t)^{m-2}(t+m)
  \]
  follows from \(-m\leq-1/m\leq0\) after cancelling the common factor.
  Taking \(l\)-fold Hadamard products and using
  \eqref{eq:garloffWagnerTwoPair} gives
  \[
    \Lambda_m^{(l)}(t)\preceq t^m\Lambda_m^{(l)}(1/t).
  \]
  The same relation holds at \(m=1\), where it reads \(1\preceq t\).
  Combining it with
  \(\widetilde E_m^{(l)}\preceq t\widetilde E_m^{(l)}\) and the
  Hadamard decompositions from the preceding proof yields
  \[
    (\theta+1)^lE_m^{(l)}(t)
      \preceq t^m\bigl((\theta+1)^lE_m^{(l)}\bigr)(1/t).
  \]
  The two polynomials on this line sum to \(E_{m+1}^{(l)}\), so the cone
  property gives the desired comparison.

  To preserve its orientation, suppose \(f,g\in\mathcal{PF}\) are nonzero
  and \(f\preceq g\).  By the Hermite--Biehler characterization,
  \(g(t)+zf(t)\) is stable.  Multiplication by \(t\) and differentiation
  in \(t\) preserve stability, and give
  \[
    \frac{\partial}{\partial t}\bigl(t(g(t)+zf(t))\bigr)
      =(\theta+1)g(t)+z(\theta+1)f(t).
  \]
  This polynomial is nonzero, and both coefficient polynomials have
  nonnegative coefficients and positive leading coefficients.  Hence
  \((\theta+1)f\preceq(\theta+1)g\).  Applying this implication
  \(l(n-m-1)\) times and using
  \eqref{eq:superEulerianPrefixOperator} proves the proposition.
\end{proof}

\section{Two further families}
\label{sec:furtherFamilies}

\subsection{Increasing runs in ternary words}

A \defin{strictly increasing run} of a word is a maximal consecutive factor
whose letters increase from left to right.  For a ternary word \(w\), we let
\(\defin{\run(w)}\) be its number of such factors, and define
\[
  \defin{W_n(t)}
  \coloneqq\sum_{w\in\{0,1,2\}^n}t^{\run(w)}
  \qquad(n\geq1),
\]
with \(\defin{W_0(t)}\coloneqq1\).

The coefficient triangle is indexed by OEIS \oeis{A120987}
\cite{OEISA120987}.  E.~Deutsch posed the original ternary-word enumeration
conjecture there, and G.~Cabrele proved it together with its extension to
arbitrary finite alphabets \cite{Cabrele2015}.  Cabrele's generating function
specializes here to
\begin{equation}
  \defin{\mathcal W(z,t)}
  \coloneqq\sum_{n\geq0}W_n(t)z^n
  =\frac{1}{1-3tz-3t(1-t)z^2-t(1-t)^2z^3}.
  \label{eq:ternaryRunGeneratingFunction}
\end{equation}

\begin{theorem}
  \label{thm:ternaryRuns}
  Every polynomial \(W_n(t)\) is real-rooted.  Moreover, the zeros of
  \(W_n(t)\) and \(W_{n+1}(t)\) weakly interlace for every \(n\geq0\).
\end{theorem}

\begin{proof}
  Differentiating the rational function in
  \eqref{eq:ternaryRunGeneratingFunction} gives
  \[
    \bigl(1-(4t-1)z\bigr)\frac{\partial \mathcal W}{\partial z}
    =3t\mathcal W+3t(1-t)\frac{\partial \mathcal W}{\partial t}.
  \]
  Comparing coefficients of \(z^n\), we obtain
  \begin{equation}
    (n+1)W_{n+1}(t)
    =\bigl(-n+(4n+3)t\bigr)W_n(t)
      +3t(1-t)W_n'(t).
    \label{eq:ternaryRunDifferentialRecurrence}
  \end{equation}
  The polynomial \(W_n\) has nonnegative coefficients and degree \(n\): its
  leading coefficient counts the weakly decreasing ternary words of length
  \(n\), and is therefore positive.

  We proceed by induction.  The initial pair \(W_0=1\), \(W_1=3t\) has
  the required property.  For \(n\geq1\), assume that \(W_n\) is
  real-rooted.  Its nonnegative coefficients force all its zeros to be
  nonpositive.  The coefficient \(3t(1-t)/(n+1)\) of \(W_n'\) in
  \eqref{eq:ternaryRunDifferentialRecurrence} is therefore nonpositive at
  every zero of \(W_n\).  Rolle's theorem gives \(W_n'\preceq W_n\).
  Since \(\deg W_{n+1}=\deg W_n+1\) and both leading coefficients are
  positive, the interlacing criterion of L.~L.~Liu and Y.~Wang
  \cite[Theorem~2.3]{LiuWang2007} applies with this sign.  It
  shows that \(W_{n+1}\) is real-rooted and that \(W_n\) interlaces
  \(W_{n+1}\).  This closes the induction and proves the theorem.
\end{proof}

\subsection{Peaks in permutations}

For \(\pi=\pi_1\dotsm\pi_n\in\mathfrak S_n\), an index
\(2\leq i\leq n-1\) is a \defin{peak} if
\(\pi_{i-1}<\pi_i>\pi_{i+1}\).  We let \(\defin{\pk(\pi)}\) be the number of
peaks and define
\[
  \defin{K_n(t)}
  \coloneqq
  \sum_{\pi\in\mathfrak S_n}t^{\pk(\pi)}
  \qquad(n\geq1).
\]
The coefficient triangle is indexed by OEIS \oeis{A008303}
\cite{OEISA008303}.
Warren and Seneta obtained the recurrence
\begin{equation}
  K_n(t)
  =\bigl(2+(n-2)t\bigr)K_{n-1}(t)
   +2t(1-t)K_{n-1}'(t),
  \qquad n\geq2,
  \label{eq:peakRecurrence}
\end{equation}
and used it to prove real-rootedness and consecutive interlacing
\cite{WarrenSeneta1996}.  The following theorem restates their result in the
parity-sensitive form used below.

\begin{theorem}
  \label{thm:permutationPeaks}
  For every \(n\geq1\), the polynomial \(K_n(t)\) is real-rooted.  For
  \(n\geq3\), it has \(\lfloor(n-1)/2\rfloor\) simple negative zeros, and
  the zeros of \(K_n(t)\) and \(K_{n+1}(t)\) strictly alternate.
\end{theorem}

\begin{proof}
  We have \(K_1=1\), \(K_2=2\), and \(K_3=4+2t\).  For \(n\geq4\), suppose
  that \(K_{n-1}\) has degree \(d=\lfloor(n-2)/2\rfloor\), positive leading
  coefficient, and simple negative zeros.  At any such zero \(r\), recurrence
  \eqref{eq:peakRecurrence} gives
  \[
    K_n(r)=2r(1-r)K_{n-1}'(r).
  \]
  These values alternate in sign, while \(K_n(0)=2K_{n-1}(0)>0\).  Hence
  there is one new zero between each consecutive pair and one between the
  rightmost zero and \(0\).  The coefficient of \(t^{d+1}\) in \(K_n\) is
  \((n-2-2d)\) times the leading coefficient of \(K_{n-1}\).  It vanishes
  when \(n\) is even, and is positive when \(n\) is odd; in the latter case,
  comparison at \(-\infty\) supplies the remaining leftmost zero.  Degree
  counting makes all these zeros simple and proves the stated strict
  alternation.  In the even case, the positive value at zero and the negative
  zeros also force the leading coefficient to be positive; in the odd case,
  this follows from the displayed top coefficient.  The induction closes.
\end{proof}

We next refine by peak values.  If
\(\defin{\operatorname{PKV}(\pi)}\) is the set of values at the peaks of
\(\pi\), P.~Alexandersson and O.~Nabawanda considered
\[
  \defin{\widehat K_n(\boldsymbol{x})}
  \coloneqq
  \sum_{\pi\in\mathfrak S_n}
  \prod_{j\in\operatorname{PKV}(\pi)}x_j.
\]
Its diagonal specialization is \(K_n(t)\).  The stability of
\(\widehat K_n\), together with weighted consecutive interleaving, was
conjectured in \cite[Conjecture~4.1]{AlexanderssonNabawanda2022}.

\begin{theorem}
  \label{thm:peakValueStability}
  The polynomial \(\widehat K_n(\boldsymbol{x})\) is stable for every
  \(n\geq1\).  Moreover, for every positive sequence
  \(\lambda_1,\lambda_2,\dotsc\) and every \(n>1\),
  \[
    \widehat K_{n-1}(\lambda_1t,\dotsc,\lambda_{n-1}t)
    \preceq
    \widehat K_n(\lambda_1t,\dotsc,\lambda_nt).
  \]
\end{theorem}

\begin{proof}
  We first prove stability by induction on \(n\).  The induction step uses
  ordinary homogenization to rewrite the insertion recurrence as a stable
  product followed by differentiation and real specialization.  Alexandersson
  and Nabawanda's insertion recurrence
  \cite[Proposition~4.1]{AlexanderssonNabawanda2022} is
  \begin{equation}
    \widehat K_n
    =\bigl(2+(n-2)x_n\bigr)\widehat K_{n-1}
     +2x_n\sum_{j=3}^{n-1}(1-x_j)\partial_{x_j}\widehat K_{n-1}.
    \label{eq:peakValueRecurrence}
  \end{equation}
  We put
  \[
    \defin{\widetilde K_n(\boldsymbol{u})}
    \coloneqq\widehat K_n(\boldsymbol{1}+\boldsymbol{u}).
  \]
  The polynomial \(\widetilde K_{n-1}\) has nonnegative coefficients and
  degree \(d=\lfloor(n-2)/2\rfloor\).  Suppose inductively that it is stable,
  and decompose it into homogeneous parts
  \[
    \widetilde K_{n-1}=\sum_{k=0}^{d}\widetilde K_{n-1,k}.
  \]
  Equation~\eqref{eq:peakValueRecurrence} gives
  \begin{equation}
    \widetilde K_n=\sum_{k=0}^{d}
      \bigl((n-2k)+(n-2-2k)u_n\bigr)\widetilde K_{n-1,k}.
    \label{eq:translatedPeakValueRecurrence}
  \end{equation}
  Hence the ordinary degree-\(d\) homogenization
  \[
    \defin{\mathcal H(y,\boldsymbol{u})}
    \coloneqq\sum_{k=0}^{d}
      y^{d-k}\widetilde K_{n-1,k}(\boldsymbol{u})
  \]
  is stable by \cite[Theorem~4.5]{BorceaBrandenLiggett2009}.  A direct
  coefficient comparison with
  \eqref{eq:translatedPeakValueRecurrence} now gives
  \[
    \widetilde K_n=
    \begin{cases}
      \displaystyle
      2\left.\partial_y\bigl((y+u_n)\mathcal H\bigr)\right|_{y=1},
        & n\text{ even},\\[6pt]
      \displaystyle
      \left.\partial_y
        \bigl((y+u_n)(y+1)\mathcal H\bigr)\right|_{y=1},
        & n\text{ odd}.
    \end{cases}
  \]
  The factors \(y+u_n\) and \(y+1\) are stable.  Products,
  differentiation, and specialization at a real value preserve stability or
  yield the zero polynomial.  Since \(\widetilde K_n(\boldsymbol{0})=n!\),
  the latter alternative cannot occur.  The identities
  \(\widetilde K_1=1\) and \(\widetilde K_2=2\) begin the induction, and
  translating the variables back proves the first assertion.

  For the second assertion, assume \(n\geq3\), since \(n=2\) is immediate.
  We fix positive \(\lambda_j\), and write
  \[
    \defin{F(t,s)}
    \coloneqq\widehat K_n(\lambda_1t,\dotsc,
      \lambda_{n-1}t,\lambda_ns)
    =\defin{A_0(t)}+s\defin{A_1(t)}.
  \]
  This bivariate specialization is stable, and
  \(A_0=2\widehat K_{n-1}(\lambda_1t,\dotsc,\lambda_{n-1}t)\) by
  \eqref{eq:peakValueRecurrence}.  Hermite--Biehler gives
  \(A_1\preceq A_0\).  The displayed identity shows that \(A_0\) is nonzero,
  while the permutation \(1,n,2,3,\dotsc,n-1\) shows that \(A_1\) is nonzero.
  Their nonnegative coefficients therefore give positive leading
  coefficients, and their real zeros are nonpositive.  Wagner's
  multiplication rule
  \cite[Lemma~3.2(iii)]{AlexanderssonNabawanda2022} gives
  \(A_0\preceq tA_1\), and the cone rule
  \cite[Lemma~3.2(ii)]{AlexanderssonNabawanda2022} then gives
  \[
    A_0\preceq A_0+tA_1=F(t,t).
  \]
  Positive rescaling does not change the roots, so this is the asserted
  relation.  This completes the proof.
\end{proof}

% Related Lean: lean/AlexanderssonFormalizations/Papers/EulerianCompendium/
% EulerianCompendium/ZigzagDefinitions.lean; generic operators in RealRooted.
\section{Zig-zag Eulerian polynomials}
\label{sec:zigzag}

We study the swap statistic introduced by Coons and
Sullivant~\cite[Definition~3 and Theorem~13]{CoonsSullivant2023},
using the recurrence of Petersen and Zhuang~\cite{PetersenZhuang2024}.
Our proof separates the combinatorial identification from a general theorem
about the location of polynomial roots. We first express the swap polynomial,
after a fractional linear change of variable, as a differential-operator
transform of a centered polynomial. We then bound the complex roots of the
centered polynomial in a fixed disk. Finally, we prove that the transform sends
every nonzero real even or odd polynomial of degree at least three with roots
in this disk to a polynomial with simple roots in the interval \((-1,1)\).
The inverse change of variable gives the desired negative roots. This operator
viewpoint is closely related to Borcea--Br\"and\'en theory
\cite{BorceaBranden2009}, which characterizes linear operators preserving
nonvanishing in prescribed circular domains through their symbols. Here the
transfer theorem is proved directly by inserting real linear factors and
conjugate pairs, using interlacing and polynomial inequalities. Its hypotheses
involve only parity, degree, and a root bound, so it applies independently of
the zig-zag enumeration; we state it separately in
Lemma~\ref{lem:zigzagUniversaltransfer}.

We write
\(\defin{\mathcal A_n}\coloneqq\{\sigma\in\mathfrak S_n:\sigma_1<\sigma_2>\sigma_3<\dotsb\}\), and define

\[
\defin{H_n(t)}\coloneqq\sum_{\sigma\in\mathcal A_n}t^{\operatorname{swaps}(\sigma)},
\qquad
\defin{\operatorname{swaps}(\sigma)}\coloneqq
\#\{i\in[n-1]:\sigma^{-1}(i)+1<\sigma^{-1}(i+1)\}.
\]

% Related Lean: ../../lean/AlexanderssonFormalizations/Papers/EulerianCompendium/
% EulerianCompendium/ZigzagRealRooted.lean, zigzagEulerianPolynomial_roots.
% Reusable analysis: RealRooted/CenteredEulerian/ColumnBound.lean;
% transfer: RealRooted/IntervalDifferential/UniversalTransfer.lean.
\begin{theorem}\label{thm:zigzagEulerian}
\(H_1=H_2=1\). For \(n\ge3\), \(H_n\) has degree \(n-2\) and all its roots are simple and strictly negative.

\end{theorem}

Define the centered polynomials by

\[
\sum_{n\ge0}P_n(c)\frac{s^n}{n!}
=\sqrt{1-s^2}\bigl(\sec(c\arcsin s)+\tan(c\arcsin s)\bigr).
\]

Put \(L=(z^2-1)\frac{d}{dz}+3z\) and \(\rho=13/3\). We divide the proof into three lemmas, established below.

\begin{lemma}\label{lem:zigzagIdentification}
 For \(n\ge2\), \(c^2-1\mid P_n(c)\), and \(Q_n=P_n/(c^2-1)\) is a real even or odd polynomial of degree \(n-2\). For \(z\ne-1\),

\begin{equation}
(z+1)^{n-2}H_n\!\left(\frac{z-1}{z+1}\right)
=\frac2{n!}\,Q_n(L)1.
\label{eq:zigzagI}
\end{equation}
The left side extends to a polynomial after clearing the denominator.
\end{lemma}

\begin{lemma}\label{lem:zigzagCentereddisk}
 For every \(n\ge0\), each complex root \(c\) of \(P_n\) satisfies \(|c|^2<\rho\).

\end{lemma}

\begin{lemma}\label{lem:zigzagUniversaltransfer}
 If a nonzero real even or odd polynomial \(Q\) has degree at least three and every root satisfies \(|c|^2\le\rho\), then \(Q(L)1\) has simple roots in \((-1,1)\).

\end{lemma}

\begin{proof}[Proof of Theorem~\ref{thm:zigzagEulerian}] For \(n\ge5\), put \(D=n-2\), so
\(Q_n\) has real parity and degree \(D\ge3\). By Lemmas~\ref{lem:zigzagCentereddisk}
and~\ref{lem:zigzagUniversaltransfer}, the degree-\(D\) polynomial \(Q_n(L)1\) has \(D\) simple
roots in \((-1,1)\); here \(L\) raises degree by exactly one. Identity~\eqref{eq:zigzagI} gives
\(\deg H_n\le D\), and the bijection \(t=(z-1)/(z+1)\) supplies its \(D\) simple negative roots.
The remaining polynomials are

\[
H_1=H_2=1,\qquad H_3=1+t,\qquad H_4=1+3t+t^2.
\]

Their roots have the stated properties.
\end{proof}

\subsection{Identification by finite polynomial
identities}\label{identification-by-finite-polynomial-identities}

We prove Lemma~\ref{lem:zigzagIdentification} from the swap definition and
the centered generating function. All series below are formal.

\subsubsection{A scalar recurrence for weak alternating
words}\label{a-scalar-recurrence-for-weak-alternating-words}

Let \(\Omega_n(k)\) count words \(w_1\le w_2\ge w_3\le\dotsb\)
of length \(n\) in \([k]\), with
\(\Omega_0(k)=1\), and put \(F_k(x)=\sum_{n\ge0}\Omega_n(k)x^n\), so \(F_0=1\).

Complementing letters equates the two orientations. Split at the first \(k\). For down-up
and up-down words, the preceding lengths are respectively even and odd, except when an
up-down word starts with \(k\). That exception is either \(k\), or \(kk\) followed by an
arbitrary up-down word. Adding these decompositions gives

\begin{equation}
2F_k=2F_{k-1}+xF_{k-1}F_k+x+x^2F_k,
\qquad
F_k=\frac{2F_{k-1}+x}{2-x^2-xF_{k-1}}.
\label{eq:zigzagA1}
\end{equation}

This is \cite[Proposition~3.17 and Theorem~3.15]{PetersenZhuang2024}.

Put \(s=\sin\theta\), \(q=\cos\theta\), and \(Y_k=(F_k(2s)+s)/q\). Recurrence
\eqref{eq:zigzagA1} becomes

\begin{equation}
Y_0=\frac{1+s}{q},\qquad
Y_k=\frac{qY_{k-1}+s}{q-sY_{k-1}}.
\label{eq:zigzagA2}
\end{equation}

The addition formula for \(\mathcal E(u)=\sec u+\tan u\) is
\(\mathcal E(u+2\theta)=(q\mathcal E(u)+s)/(q-s\mathcal E(u))\). Thus
\(Y_k=\mathcal E((2k+1)\theta)\) by induction. Setting \(\theta=\arcsin s\) and comparing
coefficients proves

\begin{equation}
P_n(2k+1)=2^n n!\,\Omega_n(k)
\qquad(n\ge2,\ k\ge0).
\label{eq:zigzagA3}
\end{equation}

All denominators are formal units. The relation \(F_k(2s)=qY_k-s\) explains the degree-one
correction.

\subsubsection{Standardization gives a finite factor
expansion}\label{standardization-gives-a-finite-factor-expansion}

Naturally label the fence \(p_1<p_2>p_3<\dotsb\) as follows. Split its positions into runs
\(\{1,2\},\{3,4\},\dotsc\), with a final singleton if needed. Earlier runs receive larger
labels, and labels increase within each run. An alternating permutation \(\sigma\) gives
the linear extension whose successive positions are

\[
\sigma^{-1}(1),\dotsc,\sigma^{-1}(n).
\]

Its label word has a descent at \(i\) exactly when \(i+1\) lies in a later run than \(i\).
Such positions cannot be adjacent across a run boundary, since that boundary is a descent
of \(\sigma\). Within a run the labels increase. Thus these descents are precisely the
swaps of \(\sigma\).

Sorting a weak order-preserving map first by its value and then by increasing label gives
a unique linear extension. If its label word has \(d\) descents, the map values form a
weakly increasing sequence in \([k]\), with a strict increase at each descent. Hence

\begin{equation}
\Omega_n(k)=\sum_{\sigma\in\mathcal A_n}
\binom{k+n-1-\operatorname{swaps}(\sigma)}{n}.
\label{eq:zigzagA4}
\end{equation}

For \(n\ge2\), every swap number satisfies \(0\le d\le n-2\): otherwise all \(n-1\)
successive position gaps would be at least two. Define

\begin{equation}
B_{a,b}(c)=
\prod_{j=0}^{a-1}(c-(2j+3))
\prod_{j=0}^{b-1}(c+(2j+3)).
\label{eq:zigzagA5}
\end{equation}

For integers \(k\ge n\), expanding the \(n\) consecutive factors in the binomial
coefficient gives

\[
2^n n!\binom{k+n-1-d}{n}
=\bigl((2k+1)^2-1\bigr)
B_{d,n-2-d}(2k+1).
\]

Combining this with \eqref{eq:zigzagA3}--\eqref{eq:zigzagA4}, and using polynomial
equality at infinitely many odd integers, proves the finite identity

\begin{equation}
P_n(c)=(c^2-1)
\sum_{\sigma\in\mathcal A_n}
B_{\operatorname{swaps}(\sigma),\,n-2-\operatorname{swaps}(\sigma)}(c).
\label{eq:zigzagA6}
\end{equation}

In particular \(Q_n=P_n/(c^2-1)\) has degree \(n-2\): each summand is monic of that
degree, and \(\mathcal A_n\ne\varnothing\). The defining centered series gives
\(P_n(-c)=(-1)^nP_n(c)\), so \(Q_n\) has the required parity.

This standardization is the fence interpretation discussed by Petersen and
Zhuang~\cite{PetersenZhuang2024}; the argument above fixes the precise
swap convention.

\subsubsection{The finite operator identity}\label{the-finite-operator-identity}

Put \(F_{a,b}(z)=(z-1)^a(z+1)^b\). Direct differentiation gives

\[
\begin{aligned}
(L-(2a+3))F_{a,b}&=(a+b+3)F_{a+1,b},\\
(L+(2b+3))F_{a,b}&=(a+b+3)F_{a,b+1}.
\end{aligned}
\]

Starting with \(B_{0,0}(L)1=1\), induction on \(a+b\) therefore proves the reusable
identity

\begin{equation}
2B_{a,b}(L)1=(a+b+2)!\,(z-1)^a(z+1)^b.
\label{eq:zigzagA7}
\end{equation}

Apply \eqref{eq:zigzagA7} termwise to \eqref{eq:zigzagA6}. Writing \(D=n-2\) and
\(d=\operatorname{swaps}(\sigma)\), we obtain

\[
\frac2{n!}Q_n(L)1
=\sum_{\sigma\in\mathcal A_n}(z-1)^d(z+1)^{D-d}
=(z+1)^D H_n\!\left(\frac{z-1}{z+1}\right).
\]

This proves Lemma~\ref{lem:zigzagIdentification}, including
\(\deg H_n\le n-2\). The case \(n=1\) remains the separate constant \(H_1=1\).

\subsection{The centered root disk}\label{the-centered-root-disk}

We prove Lemma~\ref{lem:zigzagCentereddisk} with \(\rho=13/3\). Put \[
p_n(c)=\frac{P_n(c)}{n!},\qquad
q(s)=\sqrt{1-s^2},\qquad f(s)=\arcsin s,\qquad
a=\frac{\pi^2}{4}.
\] Here \(E_n=n![z^n](\sec z+\tan z)\). The defining generating function shows that
\(p_n\) has parity \(n\), degree \(n\), and leading coefficient \(E_n/n!>0\). At
\(c=\pm1\) it becomes \(1\pm s\), so \(c^2-1\mid p_n\) for \(n\ge2\).

\subsubsection{A finite weighted-sum criterion}\label{a-finite-weighted-sum-criterion}

Put \(r_i(c)=p_{i+2}(c)/(c^2-1)\). These polynomials form a graded basis, so there are
unique coefficients \(R_{ij}\) such that \[
c^2r_i=\sum_{j=0}^{i+2}R_{ij}r_j,\qquad
R_{ij}=0\quad\text{if }i\not\equiv j\pmod2.
\] We will prove the column bounds \begin{equation}
\sum_{i\ge0}\omega_i|R_{ij}|<\rho\,\omega_j,
\qquad
\omega_0=\omega_1=1,\quad \omega_j=2\ (j\ge2).       \label{eq:zigzagB1}
\end{equation}

To deduce the root bound from \eqref{eq:zigzagB1}, write \[
r_{2d+\epsilon}(c)=c^\epsilon q_d^\epsilon(c^2),
\qquad \epsilon\in\{0,1\}.
\] The polynomial \(q_d^\epsilon\) has degree \(d\), and \(q_0^0=1/2,\ q_0^1=1/3\).
Suppose \(N\ge1\) and \(q_N^\epsilon(u)=0\), and put \(v_j=q_j^\epsilon(u)\) for
\(0\le j<N\). Evaluating the recurrence gives the finite identities \[
uv_i=\sum_{j=0}^{N-1}R_{2i+\epsilon,\,2j+\epsilon}v_j.
\] Set \(V=\sum_{j<N}\omega_{2j+\epsilon}|v_j|>0\); positivity follows already from
\(v_0\ne0\). The triangle inequality and \eqref{eq:zigzagB1} give \[
\begin{aligned}
|u|V
&\le\sum_{j<N}|v_j|
     \sum_{i<N}\omega_{2i+\epsilon}
       |R_{2i+\epsilon,\,2j+\epsilon}|\\
&<\rho\sum_{j<N}\omega_{2j+\epsilon}|v_j|
=\rho V.
\end{aligned}
\] Thus \(|u|<\rho\). The remaining possible roots of \(p_n\) are \(0,\pm1\), which
satisfy the same bound.

\subsubsection{A scalar estimate for the Euler
ratios}\label{a-scalar-estimate-for-the-euler-ratios}

Let \[
b_k=(k+1)(k+2)\frac{E_k}{E_{k+2}},\qquad
t_m=\frac{(-1)^{(m-1)/2}}m\quad(m\text{ odd}),
\] and, for integers \(s\ge3\), put \[
S_s=\sum_{\substack{m\ge3\\m\text{ odd}}}m^{-s},
\qquad
T_s=\sum_{\substack{m\ge3\\m\text{ odd}}}t_m^s.
\] The classical secant/tangent partial fractions give \begin{equation}
\frac{E_k}{k!}
=2\left(\frac2\pi\right)^{k+1}(1+T_{k+1}),
\qquad
\frac{b_k}{a}=\frac{1+T_{k+1}}{1+T_{k+3}}
\quad(k\ge2).                                    \label{eq:zigzagB2}
\end{equation} The partial fractions follow from the shifted cotangent and cosecant
formulas in \cite[Equations~4.22.3 and~4.22.5]{DLMF}.

The one summation estimate is \[
\frac1{m^3}
<
\frac14\left(\frac1{(m-1)^2}-\frac1{(m+1)^2}\right),
\] which telescopes over odd \(m\ge3\). Hence \begin{equation}
S_3<\frac1{16},\qquad S_5\le\frac{S_3}{9}<\frac1{144}.       \label{eq:zigzagB3}
\end{equation} In particular, all series above converge absolutely, and \(S_s<1\) for
\(s\ge3\). Since \(t_m^{s+2}=t_m^s/m^2\), \[
|T_s-T_{s+2}|\le S_s-S_{s+2}.
\] Therefore \eqref{eq:zigzagB2} gives \begin{equation}
|b_k-a|
\le a\,\frac{S_{k+1}-S_{k+3}}{1-S_{k+3}}
\le aS_{k+1}
\qquad(k\ge2).
\label{eq:zigzagB4}
\end{equation} The last inequality uses only \(S_{k+1}\le1\).

We will use the elementary upper bound \begin{equation}
a<\frac52.                                       \label{eq:zigzagB5}
\end{equation} Indeed, \(T_4>T_6\ge0\), so \eqref{eq:zigzagB2} at \(k=3\) gives
\(a<b_3=5/2\). The two exceptional ratios are \(b_0=2,b_1=3\).

\subsubsection{An exact change-of-basis formula}\label{an-exact-change-of-basis-formula}

Set \[
h(s)=\left(\frac{s}{f(s)}\right)^2,\qquad
w(z)=\frac{\sec z-\cos z}{z^2},\quad w(0)=1,
\] \[
\Phi_j(z)=w(z)\sin^jz=\sum_{k\ge0}\phi_{jk}z^k,
\qquad
\Delta_j(z)=\sum_{k\ge2}(b_k-a)\phi_{jk}z^k.
\] The generating function of column \(j\) is \begin{equation}
\begin{aligned}
C_j(s)\coloneqq\sum_{i\ge0}R_{ij}s^i
=B^2\Bigl[
&a\,h(s)s^j
+(2-a)\delta_{j0}q(s)
+(3-a)\delta_{j1}q(s)f(s)\\
&+q(s)\Delta_j(f(s))
\Bigr],
\end{aligned}
\label{eq:zigzagB6}
\end{equation} where \(B(\sum g_is^i)=\sum g_{i+1}s^i\).

Here is a derivation using only linear functionals on polynomials. Define \(\lambda_j\) on
monomials by \[
\lambda_j(c^k)=\frac{k!}{E_k}[z^k](\sec z\,\sin^jz).
\] The generating function gives \[
\sum_{n\ge0}\lambda_j(p_n)s^n
=q(s)\sec(f(s))\sin^j(f(s))=s^j.
\] Thus \(\lambda_j(p_n)=\delta_{jn}\). Multiplying the defining recurrence for \(r_i\) by
\(c^2-1\) now gives \(R_{ij}=\lambda_{j+2}(c^2p_{i+2})\). Applying \(\lambda_{j+2}\) to
\(c^2\) times the generating function, and using \[
[z^{k+2}](\sec z\,\sin^{j+2}z)=\phi_{jk},
\] yields \[
C_j=B^2\left(q\sum_{k\ge0}b_k\phi_{jk}f^k\right).
\] Now \(\phi_{j0}=\delta_{j0}\), \(\phi_{j1}=\delta_{j1}\), and
\(q(s)\Phi_j(f(s))=h(s)s^j\). Separating \(b_0,b_1\) proves \eqref{eq:zigzagB6}. All these
identities are initially formal, with finite sums in each coefficient.

\subsubsection{Coefficient signs and rational
majorants}\label{coefficient-signs-and-rational-majorants}

Use \(\|G\|_1=\sum_i|[s^i]G|\). The three main series can be written using nonnegative sequences
\(\eta_r,d_r,e_r\) as follows: \begin{equation}
\begin{array}{c|c|c}
\text{Series}&\text{First coefficient}&\text{Coefficient sum}\\ \hline
q(s)=1-\sum_{r\ge1}\eta_rs^{2r}
&\eta_1=1/2&\sum\eta_r=1\\
q(s)f(s)=s-\sum_{r\ge1}d_rs^{2r+1}
&d_1=1/3&\sum d_r=1\\
h(s)=1-\sum_{r\ge1}e_rs^{2r}
&e_1=1/3&\sum e_r=1-1/a.
\end{array}                                      \label{eq:zigzagB7}
\end{equation} For the first row, \((1-s^2)q'=-sq\) gives \(\eta_1=1/2\) and
\(\eta_{r+1}=(2r-1)\eta_r/(2r+2)\) for \(r\ge1\). Thus \(q=1\) minus a series with
nonnegative coefficients, so \(1/q\), \(f\), and \(f/q\) have nonnegative coefficients,
using \(f'=1/q\) and \(f(0)=0\). The second row follows from \((qf)'=1-sf/q\). The first
two sums follow by monotone convergence as \(s\uparrow1\), using \(q(1)=0\) and
\(f(1)=\sqrt a\). In particular, \(\|q\|_1=2\) and \(\|f\|_1=\sqrt a\).

For the third row, write \[
\frac{f(s)^2}{s^2}=\sum_{n\ge0}A_nx^n,\qquad x=s^2.
\] The identity \((1-s^2)(f^2)''-s(f^2)'=2\) gives \(A_0=1\) and \[
A_{n+1}=r_nA_n,\qquad
r_n=\frac{2(n+1)^2}{(n+2)(2n+3)}.
\] These ratios are positive and nondecreasing. For any sequence with \(A_0=1\) and
nondecreasing positive ratios, the coefficients of its reciprocal \(1-\sum_{n\ge1}e_nx^n\)
satisfy \[
e_1=A_1,\qquad
e_{n+1}=\sum_{i=1}^{n}e_iA_{n-i}(r_n-r_{n-i})
\quad(n\ge1).
\] Indeed, the product identity is \(A_n=\sum_{i=1}^ne_iA_{n-i}\). Subtract
\(r_n\) times its instance at \(n\) from its instance at \(n+1\).
The displayed recurrence gives \(e_n\ge0\) by
induction; the product identity then gives \(e_n\le A_n\). In our case, \(e_1=A_1=1/3\)
and \(\sum A_n=a\), so the reciprocal is absolutely summable. Evaluating its product with
\(\sum A_nx^n\) at \(x=1\) proves \(\sum e_n=1-1/a\). Since this sum is at least
\(e_1=1/3\), we also obtain \(a\ge3/2\). Together with \eqref{eq:zigzagB5}, this gives
\(|2-a|\le1/2\).

The coefficients of \(w\) are nonnegative: \[
[z^{2r}]w(z)
=\frac{E_{2r+2}-(-1)^{r+1}}{(2r+2)!}\ge0.
\] Use the rational series \[
U(z)=\frac{z}{1-z^2/6}
=\sum_{r\ge0}\frac{z^{2r+1}}{6^r}.
\] The elementary induction \((2r+1)!\ge6^r\) shows that the absolute coefficients of
\(\sin z\) are bounded coefficientwise by those of \(U\). Since \(w\) has nonnegative
coefficients, for nonnegative \(z\) inside the convergence disks, \begin{equation}
\sum_{k\ge2}|\phi_{jk}|z^k\le M_j(z),\qquad
M_j(z)=
\begin{cases}
w(z)-1,&j=0,\\
w(z)U(z)-z,&j=1,\\
w(z)U(z)^j,&j\ge2.
\end{cases}                                      \label{eq:zigzagB8}
\end{equation} For \(j=1\), the removed linear coefficient is exactly \(1\), so
subtracting \(z\) is legitimate in this coefficient bound.

We need only the following bounds for \(z=\sqrt a\,t\), \(0\le t\le1/3\): \begin{equation}
w(z)<\frac65,\qquad
0\le w(z)-1\le\frac32t^2,\qquad
0\le U(z)\le\frac53t<1.                           \label{eq:zigzagB9}
\end{equation} For \(t>0\), the two nonconstant bounds may be taken strict. Indeed
\(z^2<1/3\), and the elementary estimates \(\sin^2z\le z^2\), \(\cos z\ge1-z^2/2\) give \[
w(z)=\frac{\sin^2z}{z^2\cos z}
\le\frac1{1-z^2/2}<\frac65,
\] \[
w(z)-1\le\frac{z^2}{2-z^2}
<\frac35z^2<\frac32t^2.
\] The case \(z=0\) follows from \(w(0)=1\). Also \(z^2/6<1/20\) and \(\sqrt a<19/12\), by
\(a<5/2\). Consequently \[
U(z)<\frac{20}{19}z<\frac53t\qquad(t>0).
\]

\subsubsection{Two column bounds}\label{two-column-bounds}

Let \(\|G\|_\omega=\sum_i\omega_i|[s^i]G|\) and \(\gamma_j=\|C_j\|_\omega/\omega_j\). The
first three terms in \eqref{eq:zigzagB6} form the main part. From \eqref{eq:zigzagB7},
\(aB^2(hs^j)\) has normalized column norm \[
\begin{cases}
5a/3-2,&j=0,1,\\
3a/2-1,&j=2,3,\\
2a-1,&j\ge4.
\end{cases}
\] The two corrections contribute at most \(\tfrac32|2-a|\le3/4\) in column \(0\) and
\(\tfrac53(3-a)\) in column \(1\). Thus the main-part bounds are \begin{equation}
\begin{cases}
5a/3-5/4<3,&j=0,\\
3,&j=1,\\
3a/2-1<3,&j=2,3,\\
2a-1<4,&j\ge4.
\end{cases}                                      \label{eq:zigzagB10}
\end{equation}

It remains to bound the last term \(B^2[q\Delta_j(f)]\). Because \(f\) has nonnegative
coefficients with sum \(\sqrt a\), the triangle inequality and \eqref{eq:zigzagB4} give
\begin{equation}
\begin{aligned}
\|\Delta_j(f)\|_1
&\le\sum_{k\ge2}|b_k-a|\,|\phi_{jk}|a^{k/2}\\
&\le a\sum_{\substack{m\ge3\\m\text{ odd}}}\frac1m
       \sum_{k\ge2}|\phi_{jk}|
       \left(\frac{\sqrt a}{m}\right)^k\\
&\le a\sum_{\substack{m\ge3\\m\text{ odd}}}
       \frac{M_j(\sqrt a/m)}m.                    \label{eq:zigzagB11}
\end{aligned}
\end{equation} The interchange is between nonnegative sums. The bounds below prove that
they are finite; this also proves absolute summability of every residual column.

Output weights are at most \(2\), \(B^2\) cannot increase the unweighted coefficient norm,
and \(\|q\|_1=2\). Thus the normalized tail is bounded by \eqref{eq:zigzagB11} multiplied
by \(4/\omega_j\). For \(j\le3\), \eqref{eq:zigzagB8}--\eqref{eq:zigzagB9} yield the
common estimate \begin{equation}
\frac4{\omega_j}M_j(\sqrt a\,t)\le8t^2
\qquad(0\le t\le1/3).                             \label{eq:zigzagB12}
\end{equation} Here are the three checks for \(t>0\); at \(t=0\), all terms vanish:

\begin{itemize}
\item
  For \(j=0\), \(4M_0\le6t^2\).
\item
  For \(j=1\), use \[
  wU-z=\bigl((w-1)+z^2/6\bigr)U<4t^3.
  \] Hence \(4M_1<16t^3\le(16/3)t^2<8t^2\).
\item
  For \(j=2,3\), \(U<1\) and \(U^2<3t^2\), so \(2M_j\le2wU^2<(36/5)t^2<8t^2\).
\end{itemize}

Their full tails are therefore less than \begin{equation}
8aS_3<\frac54<\frac43.                            \label{eq:zigzagB13}
\end{equation} For \(j\ge4\), use \(U^j\le U^4<8t^4\) and \(2aw<6\). Their full tails are
less than \begin{equation}
48S_5<\frac13.                                   \label{eq:zigzagB14}
\end{equation} The elementary rational comparisons used here are \((5/3)^2<3\) and
\((5/3)^4<8\).

Combining \eqref{eq:zigzagB10}, \eqref{eq:zigzagB13}, and \eqref{eq:zigzagB14} gives:

\[
\begin{array}{c|c|c|c}
\text{Column}&\text{Main part}&\text{Tail}&\gamma_j\\ \hline
j\le3&\le3&<4/3&<13/3\\
j\ge4&<4&<1/3&<13/3
\end{array}
\]

This proves \eqref{eq:zigzagB1}. The finite weighted-sum argument above now proves that
every zero \(c\) of every \(P_n\) satisfies \(|c|^2<13/3\).

\subsection{The differential-operator transfer}\label{the-differential-operator-transfer}

Fix

\[
L=(z^2-1)\frac{d}{dz}+3z,
\qquad \rho=\frac{13}{3}.
\]

We prove that if a nonzero real even or odd polynomial \(Q\) has degree at least three and
all its zeros satisfy \(|c|^2\le\rho\), then \(Q(L)1\) has simple zeros in \((-1,1)\).

\subsubsection{A finite sign criterion}\label{a-finite-sign-criterion}

Suppose \(F,G\) have positive leading coefficients and degrees \(d,d+1\), and \(F\) has
simple zeros \(-1<r_1<\dotsb<r_d<1\). If

\[
G(r_i)F'(r_i)<0\quad(1\le i\le d),
\qquad (-1)^{d+1}G(-1)>0,
\qquad G(1)>0,
\]

then \(G\) has one simple zero in each interval cut out by \(-1,r_1,\dotsc,r_d,1\).
Indeed, \(F'(r_i)\) has sign \((-1)^{d-i}\), so the displayed conditions give alternating
signs at these \(d+2\) points. The intermediate value theorem supplies \(d+1\) distinct
zeros; the degree makes them all the zeros, and all simple.

At every zero \(s\) of \(G\), the product signs also give

\[
F(s)G'(s)>0:
\]

if \(s\) is the \(j\)-th zero of \(G\), both factors have sign \((-1)^{d+1-j}\).

\subsubsection{Inserting real factors and conjugate
pairs}\label{inserting-real-factors-and-conjugate-pairs}

Let \(F\) have degree \(d\), positive leading coefficient, and simple zeros in \((-1,1)\).
For real \(a\) with \(|a|<3\), put \(G=(L-a)F\). At a zero \(r\) of \(F\),

\[
G(r)F'(r)=(r^2-1)F'(r)^2<0,
\qquad G(\pm1)=(\pm3-a)F(\pm1).
\]

The sign criterion shows that \(G\) has simple interval zeros strictly interlacing those
of \(F\).

Now assume \(d\ge4\), \(a^2+b^2\le\rho\), and put

\[
T=((L-a)^2+b^2)F=(L-a)G+b^2F.
\]

The condition implies \(|a|<3\), so the preceding conclusion applies. We use two
polynomial inequalities. Set \(h=1-z^2\) and

\[
J=h(F'^2-FF'')+2zFF'-dF^2,
\qquad
K=(d-1)F'^2-dFF''.
\]

Writing \(F=(z-r_i)F_i\), the product rule gives \(F'=\sum_iF_i\) and
\(F'^2-FF''=\sum_iF_i^2\). Consequently

\[
J=\sum_i(1-r_i^2)F_i^2,
\qquad
K=\sum_{i<j}(F_i-F_j)^2\ge0.
\]

At a zero \(s\) of \(G\), strict interlacing gives \(F(s)\ne0\). Thus every
\(F_i(s)\ne0\), and \(J(s)>0\). Put \(h=1-s^2>0\). Since \(G(s)=0\),
\(hF'(s)=(3s-a)F(s)\). Expanding \(G'\) and using this equation gives the polynomial identity

\[
\begin{aligned}
-(d+2)F(s)T(s)
={}&2hJ(s)+h^2K(s)\\
&+\bigl((d+3)(a-s)^2+5d+6-(d+2)(a^2+b^2)\bigr)F(s)^2.
\end{aligned}
\]

To check the identity, we replace \(s\) by \(z\) and add \((d+1)G^2+2((d+3)z-a)FG\) to the
left; the resulting identity holds as a polynomial identity.

The first term on the right is strictly positive, and the others are nonnegative because

\[
5d+6-(d+2)(a^2+b^2)
\ge\frac23(d-4)\ge0.
\]

Hence \(F(s)T(s)<0\). Since \(F(s)G'(s)>0\), this gives \(T(s)G'(s)<0\). The endpoint
signs follow from

\[
T(1)=((3-a)^2+b^2)F(1),
\qquad
T(-1)=((3+a)^2+b^2)F(-1).
\]

The sign criterion applies again. We have proved

\[
d\ge4,\quad a^2+b^2\le\frac{13}{3}
\quad\Longrightarrow\quad
((L-a)^2+b^2)F
\text{ has simple zeros in }(-1,1).
\]

The leading coefficients in the two insertions remain positive: their multipliers are
\(d+3\) and \((d+3)(d+4)\), respectively. The strict sign above also covers equality in
the disk condition.

\subsubsection{Cubic and quartic initial cases}\label{the-seeds-only-sign-evaluations}

For a monic odd cubic \(c^3+kc\), the disk condition gives \(|k|\le\rho\). Its transform
is

\[
(L^3+kL)1=3z(20z^2+k-11).
\]

Since \(\rho<9\), the three zeros are simple and belong to \((-1,1)\).

For a monic even quartic \(Q_4(c)=c^4+kc^2+p\), put \(B(u)=u^2+ku+p\). Its roots have
modulus at most \(\rho\), so

\[
|p|\le\rho^2,\qquad B(9)>0,\qquad B(x)\ge-\rho^2\quad(x\in\mathbb R).
\]

For the last bound, if the roots \(u,v\) are real, complete the square and use
\(|u-v|/2\le\rho\); for a conjugate pair, \(B(x)\ge0\). The inequality \(B(9)>0\) follows
from \(\rho<9\).

Write \(Q_4(L)1=A(z^2)\), where

\[
A(u)=360u^2+(12k-312)u+33-3k+p.
\]

Only three signs are needed:

\[
\begin{aligned}
A(0)&=24+B(-3)\ge24-\rho^2>0,\\
A(1/4)&=p-45/2\le\rho^2-45/2<0,\\
A(1)&=B(9)>0.
\end{aligned}
\]

Here \(\rho^2<45/2<24\). The intermediate value theorem gives one zero of \(A\) in
\((0,1/4)\) and one in \((1/4,1)\). Its degree is two, so both are simple. Consequently
\(Q_4(L)1\) has four simple zeros in \((-1,1)\).

\subsubsection{Proof of the transfer lemma}\label{completion}

Scale \(Q\) to be monic. Every real parity polynomial of degree at least four has a monic
real even quartic factor: write \(Q(c)=c^\epsilon R(c^2)\), where \(\epsilon\in\{0,1\}\)
and \(\deg R\ge2\), and choose a monic real quadratic factor of \(R\). It is either
irreducible or a product of two real linear factors, counting multiplicities.

Start with this quartic seed and insert the remaining real linear factors and conjugate
pairs. The current degree is always at least four, so every insertion is covered by the
insertion step. The factors in \(L\) commute, so the resulting polynomial is \(Q(L)1\).
Degree three uses the cubic seed. This proves the transfer theorem, including repeated
input roots and roots on the disk boundary.

\subsection*{Acknowledgements}

Generative-AI tools assisted with proof search.
The author reviewed the arguments and takes
responsibility for the final text.

\raggedright
\bibliographystyle{amsalpha}
\bibliography{eulerian_real_rootedness}

\end{document}